\documentclass[12pt, twoside]{article}

\usepackage{amsmath, amssymb, amsthm, physics, url, bbm, mathrsfs, cite, xcolor, mathtools, dsfont, esint}

\usepackage[shortlabels]{enumitem}
\usepackage[title]{appendix}

\mathtoolsset{showonlyrefs}
\usepackage[utf8]{inputenc}

\definecolor{darkblue}{RGB}{32,57,231}
\usepackage[colorlinks,
	citecolor=darkblue, 
	linkcolor=darkblue, 
	urlcolor=darkblue,
	bookmarks=false,
	pdfauthor={Aobo Chen}, 
	pdftitle={Walk dimension and vanishing curve modulus in metric measure spaces}
]{hyperref}

\usepackage{orcidlink}

\numberwithin{equation}{section}
\numberwithin{figure}{section}

\newtheorem{theorem}{Theorem}[section]
\newtheorem{lemma}[theorem]{Lemma}
\newtheorem{proposition}[theorem]{Proposition}
\newtheorem{corollary}[theorem]{Corollary}

\theoremstyle{definition}
\newtheorem{definition}[theorem]{Definition}
\newtheorem{remark}[theorem]{Remark}
\newtheorem{notation}[theorem]{Notation}

\def\dist{{\mathop {{\rm dist}}}}
\def\diam{{\mathop{{\rm diam }}}}
\newcommand*{\dif}{\mathop{}\!\mathrm{d}}
\newcommand{\supp}{\mathrm{supp}}
\newcommand{\one}{\mathds{1}}

\def\sD {{\mathcal D}} \def\sE {{\mathcal E}} \def\sF {{\mathcal F}}
 \def\sH {{\mathcal H}}

 \def\bN {{\mathbb N}} 
  \def\bR {{\mathbb R}}

 \def\bZ {{\mathbb Z}}

  \def\rF {{\mathscr F}}

\def\ol{\overline}
\def\wt{\widetilde}
\def\wh{\widehat}

\newcommand{\loc}[0]{\operatorname{loc}}
\newcommand{\set}[1]{\left\{ #1 \right\}}
\newcommand{\Sett}[2]{\left\{ #1  : \, #2 \right\}}

\newcommand\restr[2]{{
		\left.\kern-\nulldelimiterspace 
		#1 
		\vphantom{\big|}
		\right|_{#2}
}} 

\def\Mod{\operatorname{Mod}}
\def\Lip{\operatorname{Lip}}
\newcommand{\carre}[1]{
	\Gamma_{p}\!\left\langle{#1}\right\rangle
		}
\newcommand{\lip}[2]{\mathrm{Lip}_{#1}{[ #2 ]}}

\font\titlefont=cmbx12 scaled 1400
	\title{\titlefont Walk dimension and vanishing curve modulus in metric measure spaces}
\date{\today}
\author{Aobo Chen}

\begin{document}

\maketitle

\vspace{-22pt}
\begin{abstract}
On a regular local $p$-Dirichlet space supporting a Poincar\'{e} inequality and a capacity upper bound, we characterize the existence of curve families of positive $p$-modulus in terms of the small-scale behaviour of the scaling function. As an application, we show that, on an Ahlfors-regular metric measure space carrying a strongly local regular Dirichlet form with sub-Gaussian heat kernel estimates, the comparison between the walk dimension and $2$ determines whether the given metric attains its Ahlfors-regular conformal dimension.
\end{abstract}

\section{Introduction}

In the Newtonian theory on metric measure spaces, the $p$-modulus of curve families provides a quantitative way to measure the abundance of rectifiable curves, and it plays an essential role in the study of Poincar\'{e} inequalities, quasiconformal geometry, and conformal dimension. On spaces containing sufficiently many curves, the existence of a curve family of positive $p$-modulus often reflects compatibility between first-order $p$-energy and the underlying metric geometry. One motivation for studying such curve families is that positive modulus can yield lower bounds for conformal dimension; see, for example, \cite[Section~5.3]{Tys99}, \cite[Theorem~15.10]{Hei01}, \cite[Theorem~1.0.1]{KL04}, and \cite[Proposition~4.1.8]{MT10}. 

In this work, we consider a metric measure space $(X,d,m)$, where $\#X\geq2$, $(X,d)$ is a locally compact separable metric space and $m$ is a nontrivial locally finite Borel-regular measure on $(X,d)$. Our first objective is to study the following modulus-vanishing phenomenon for spaces equipped with a local $p$-energy form.

\begin{definition}
A metric measure space $(X, d, m)$ is said to satisfy the \emph{vanishing $p$-modulus property} \ref{VMp}\hypertarget{VMp}{}, if \begin{equation}\label{VMp}\tag*{$\mathrm{VM}_{p}$}
\Mod_{p}(\Gamma)=0 \ \text{for every family of curves $\Gamma$ in $(X,d)$}.
\end{equation}
\end{definition}

On certain fractals, such as planer Sierpi\'{n}ski carpets, \ref{VMp} holds for every $p\in[1,\infty)$; equivalently, every curve family has zero modulus. See \cite[Théorème~5.2]{LKP04} and \cite[Proposition~4.3.3]{MT10}. Such fractals have holes at many scales, and their functional inequalities and Brownian motions often behave quite differently from those in Euclidean spaces. In particular, their natural \emph{space-time scaling exponent}, also called the \emph{walk dimension}, is typically larger than the Euclidean exponent. Informally, the walk dimension $\beta$ describes the speed of the associated diffusion: the mean exit time from a ball of radius $r$ is typically comparable to $r^\beta$; see \cite[Lemma~3.9-(c)]{Bar98}. It also governs the Hausdorff dimension of the range of the diffusion, which is given by the minimum of the Hausdorff dimension of the ambient space and $\beta$; see \cite[Lemma~3.27]{Bar98}.

Our main result places this phenomenon in the general framework of metric measure spaces equipped with local $p$-energy forms. Under suitable structural assumptions, it gives a dichotomy for the existence of curve families of positive $p$-modulus in terms of the small-scale behaviour of the scale function. The conditions appearing in the statement are recalled in Section~\ref{s.pre}.

\begin{theorem}\label{t.main}
	Let $(X,d)$ be a complete metric space. Let $p\in(1,\infty)$. Suppose there exists a regular local $p$-Dirichlet space $(X,d,m,\sE_{p},\sF_{p},\Gamma_{p})$ that satisfies \ref{VD} and $\hyperlink{PI}{\mathrm{PI}_{p}(\Psi)}$.
	\begin{enumerate}[label=\textup{({\arabic*})},align=right,leftmargin=*,topsep=5pt,parsep=0pt,itemsep=2pt]
	\item\label{it.main1} If the scale function $\Psi$ satisfies \begin{equation}\label{e.ma1}
		\liminf_{r\downarrow 0}\frac{\Psi(r)}{r^{p}}=0,
	\end{equation}
	then the condition \ref{VMp} holds for $(X,d,m)$.
	\item\label{it.main2} Suppose in addition that $(X,d)$ satisfies the chain condition, and that \hyperlink{cap<}{$\operatorname{Cap}_{p}(\Psi)_{\leq}$} holds. If the scale function $\Psi$ satisfies \begin{equation}
		\liminf_{r\downarrow 0}\frac{\Psi(r)}{r^{p}}>0,
	\end{equation}
then the condition \ref{VMp} fails for $(X,d,m)$.
\end{enumerate}
\end{theorem}

To prove part~\ref{it.main1} of Theorem~\ref{t.main}, we use the discrete approximations of the \emph{Cheeger energy} constructed in \cite{ACDM15}. The Poincar\'{e} inequality bounds the discrete energy at scale $\delta$ by $\Psi(\delta)\delta^{-p}\sE_p(u)$ for $u\in\sF_p$. Passing to the $\Gamma$-limit then shows that the Cheeger energy vanishes on $\sF_p$. The equality between modulus and capacity subsequently yields \ref{VMp}. For part~\ref{it.main2}, the Poincar\'{e} inequality, together with the capacity upper bound, gives a \emph{weak $(1,p)$-Poincar\'{e} inequality} formulated in terms of pointwise Lipschitz constants. An adaptation of the argument of Heinonen and Koskela \cite{HK98} then produces a curve family of positive $p$-modulus.

A direct consequence of Theorem~\ref{t.main} is the following triviality of the Newtonian Sobolev space, by \cite[Proposition 7.1.33]{HKST15}. In particular, if \eqref{e.ma1} holds, then $(X, d, m)$ cannot support a nontrivial Poincar\'{e} inequality formulated in terms of upper gradients, since every function has zero minimal $p$-weak upper gradient. For Sierpi\'{n}ski carpets, the nonexistence results for Poincar\'{e} inequalities can also be found in Bourdon and Pajot \cite[Proposition 4.5]{BP02}.

\begin{corollary}
Let $(X,d)$ be a complete metric space, and let
$(X,d,m,\sE_{p},\sF_{p},\Gamma_{p})$
be a regular local $p$-Dirichlet space satisfying \ref{VD} and
$\hyperlink{PI}{\mathrm{PI}_{p}(\Psi)}$. If the scale function $\Psi$
satisfies \eqref{e.ma1}, then
\begin{equation}
N^{1,p}(X;V)=L^{p}(X;V)
\end{equation}
for every vector space $V$, where $N^{1,p}(X;V)$ denotes the
\emph{$V$-valued Newtonian Sobolev space} defined in
\cite[Eq.~(7.1.26)]{HKST15}.
\end{corollary}

Our second application concerns Ahlfors-regular conformal dimension. On planer Sierpi\'{n}ski carpets, the vanishing of curve modulus leads to the fact that the Euclidean metric does not attain the Ahlfors-regular conformal dimension. The conformal geometry of Sierpi\'{n}ski carpets is also of interest in geometric group theory, since Sierpi\'{n}ski carpets arise naturally as boundaries of hyperbolic groups and are related to a carpet analogue of Cannon's conjecture; see \cite[Theorem~5 and Conjecture~6]{KK00}.

Keith and Laakso \cite{KL04} characterized attainment of the Ahlfors-regular conformal dimension by the existence of a weak tangent carrying a curve family of positive modulus. In general, however, weak tangents can be difficult to identify. Our next result bypasses this difficulty for spaces carrying a symmetric diffusion governed by two-sided Gaussian or sub-Gaussian heat kernel estimates. Examples include complete Riemannian manifolds with nonnegative Ricci curvature \cite{LY86}, Heisenberg group \cite{Li10}, certain p.c.f.\ self-similar fractals \cite{BP88,HK99}, resistance-form spaces with volume-doubling resistance metrics \cite{Kum04}, certain non-p.c.f.\ Sierpi\'{n}ski gasket-type fractals \cite{CQ24}, Laakso-type spaces \cite{Ste13,Mur25,AEBS25}, Sierpi\'{n}ski carpets \cite{BB92,BB99,Kaj23}, and certain Julia sets equipped with suitable resistance forms and measures \cite{RT10,Kig09,Kig12}.

More precisely, we prove that, on an Ahlfors-regular metric measure space carrying a strongly local Dirichlet form satisfying $\hyperlink{HKE}{\mathrm{HKE}(\beta)}$,  comparison of the walk dimension $\beta$ with $2$ \emph{determines} whether the given metric attains its Ahlfors-regular conformal dimension. The relevant definitions are recalled in Section~\ref{s.app}.

\begin{theorem}\label{t.dic}
	Let $\alpha\in[1,\infty)$. Let $(X,d,m)$ be a metric measure space that is proper, length and Ahlfors $\alpha$-regular. Let $\beta\in(1,\infty)$. Assume that there exists a conservative, regular and strongly local Dirichlet form $(\sE,\sF)$ on $L^{2}(X,m)$ that satisfies \hyperlink{HKE}{$\mathrm{HKE}(\beta)$}. Then we have \begin{equation}
		\beta>2\ \Longleftrightarrow \ \hyperlink{VMp}{\mathrm{VM}_{2}} \text{ holds}\ \Longleftrightarrow \ \dim_{\mathrm{ARC}}(X,d)<\alpha,
	\end{equation} 
	and \begin{equation}
		\beta=2\ \Longleftrightarrow \ \hyperlink{VMp}{\mathrm{VM}_{2}} \text{ fails}\ \Longleftrightarrow \ \dim_{\mathrm{ARC}}(X,d)=\alpha.
	\end{equation}
\end{theorem}
It is worth noting that the three notions appearing in Theorem~\ref{t.dic} arise from different structures. The walk dimension $\beta$ describes the scaling of the chosen Dirichlet form, or equivalently of the associated symmetric diffusion, relative to the metric and measure. The property $\hyperlink{VMp}{\mathrm{VM}_{2}}$ depends on the metric and measure, while the Ahlfors-regular conformal dimension $\dim_{\mathrm{ARC}}$ \emph{depends only on the quasisymmetric gauge of the metric}. Under the heat kernel estimate, Theorem~\ref{t.dic} shows that these three a priori different aspects characterize one another.

The proof of Theorem \ref{t.dic} combines two main ingredients. The first is the characterization of Keith and Laakso \cite{KL04}, which relates the minimality of Ahlfors-regular conformal dimension to the existence of a weak tangent carrying a curve family of positive modulus. The second is the stability of sub-Gaussian heat kernel estimates under pointed Gromov--Hausdorff convergence, proved in \cite{Che26}. As a consequence, we show that a generalized Sierpi\'{n}ski carpet equipped with the Euclidean metric is not minimal for its Ahlfors-regular conformal dimension. We also give a second, more direct proof of this application. The argument combines the \emph{Furstenberg homogeneity} of generalized Sierpi\'{n}ski carpets with the description of their weak tangents in terms of \emph{microsets}, and directly proves that every weak tangent satisfies $\hyperlink{VMp}{\mathrm{VM}_{p}}$ for every $p\in[1,\infty)$. This argument is related to the method used in \cite[Théorème~5.2]{LKP04}.
\begin{remark}
We make two comments on the assumptions and the scope of Theorem~\ref{t.dic}.
\begin{enumerate}[label=\textup{({\roman*})},align=right,leftmargin=*,topsep=5pt,parsep=0pt,itemsep=2pt]
\item By \cite[Theorems~2.1 and~2.2]{Mur25} (see also \cite[Theorem~8.5]{Che26}), there exists an Ahlfors $\alpha$-regular metric measure Dirichlet space $(X,d,m,\sE,\sF)$ satisfying the heat kernel estimates $\hyperlink{HKE}{\mathrm{HKE}(\beta)}$ if and only if
\begin{equation}
\alpha\in[1,\infty)\ \text{and}\ \beta\in[2,\alpha+1].
\end{equation}
Thus, Theorem~\ref{t.dic} applies to every admissible pair $(\alpha,\beta)$.
\item One may ask whether Theorem~\ref{t.dic} extends to general volume-doubling metric measure spaces carrying a strongly local Dirichlet form with sub-Gaussian heat kernel estimates, with the Hausdorff dimension replaced by the \emph{Assouad dimension} $\dim_{\mathrm A}$ and the Ahlfors-regular conformal dimension replaced by the \emph{conformal Assouad dimension} $\dim_{\mathrm{CA}}$ or the conformal Hausdorff dimension $\dim_{\mathrm{CH}}$; see \cite[Definitions~1.4.14,~2.2.1,~and~2.2.4]{MT10} for definitions. The answer is negative, even in the Gaussian case.

Let $\mathrm{SG}$ be the standard Sierpi\'{n}ski gasket in $\bR^{2}$, let $(\sE,\sF)$ be its standard resistance form on $\mathrm{SG}$, let $\mu$ be the Kusuoka measure, and let $\rho$ be the harmonic geodesic metric (the intrinsic metric associated with $\mu$); see \cite[Eq.~(1.1),~Definitions~2.1~and~2.11,~Theorems 2.5~and~4.2]{Kaj12}. Then $(\sE,\sF)$ is a conservative, strongly local, regular Dirichlet form on $L^{2}(\mathrm{SG},\mu)$ satisfying the Gaussian heat kernel estimates $\hyperlink{HKE}{\mathrm{HKE}(2)}$; see \cite[Theorem~6.33]{KM23}.

By \cite[Theorem~6.33]{KM23}, the metric $\rho$ is quasisymmetric to the Euclidean metric on $\mathrm{SG}$. Since conformal dimensions are quasisymmetric invariants, \cite[Proposition~2.2.6]{MT10}, together with \cite[Theorem~1.3]{TW06} and \cite[Theorem~1.6]{EB24} (see also \cite[Theorem B.8]{KS26}), yields
\begin{equation}
	\dim_{\mathrm{CH}}(\mathrm{SG},\rho)=\dim_{\mathrm{CA}}(\mathrm{SG},\rho)=\dim_{\mathrm{ARC}}(\mathrm{SG},\rho)=1.
\end{equation}
On the other hand, by \cite[Theorems~6.1 and~7.2]{Kaj12},
\begin{equation}
	\dim_{\mathrm A}(\mathrm{SG},\rho)\geq\dim_{\mathrm H}(\mathrm{SG},\rho)>1.
\end{equation}
Therefore, without a regularity assumption relating the reference measure to the given metric, Gaussian heat kernel estimates $\hyperlink{HKE}{\mathrm{HKE}(2)}$ do not generally imply that the metric attains either its conformal Assouad dimension or conformal Hausdorff dimension.
\end{enumerate}
\end{remark}

This paper is organized as follows. In Section~\ref{s.pof}, we prove Theorem~\ref{t.main}. In Section~\ref{s.app}, we recall the terminology used in Theorem~\ref{t.dic}, prove Theorem~\ref{t.dic}, and establish directly that every weak tangent of a generalized Sierpi\'{n}ski carpet satisfies \ref{VMp} for every $p\in[1,\infty)$. In Appendix~\ref{s.apdx}, we prove that Attouch--Wets convergence and pointed Gromov--Hausdorff convergence are equivalent up to passing to a subsequence. This equivalence is used in the proof of Theorem~\ref{t.dic}.

\begin{notation}
In this paper, we use the following notation and conventions.
\begin{enumerate}[label=\textup{(\roman*)},align=right,leftmargin=*,topsep=5pt,parsep=0pt,itemsep=2pt]
	\item $\bN:=\{1,2,\ldots\}$. That is $0\notin \bN$.
	\item Let $U$ and $V$ be two open subsets in a topological space. If $U$ is precompact and the closure of $U$ is contained in $V$, then we write $U\Subset V$.
	\item Let $X$ be a non-empty set. We define $\one_{A}\in\mathbb{R}^{X}$ for $A\subset X$ by
	 \[\one_{A}(x):= \begin{dcases}
	 	1 & \mbox{if $x \in A$,}\\
	 	0 & \mbox{if $x \notin A$.}
	 \end{dcases}\]
	 \item Let $X$ be a topological space. We set the class of continuous function spaces:\begin{align}
		C(X)&:=\Sett{f\in\bR^{X}}{\textrm{$f$ is a continuous real-valued function on $X$}},\\
		C_c(X)&:=\Sett{f\in C(X)}{\textrm{$X\setminus f^{-1}(0)$ has compact closure in $X$}}.
			\end{align}
		For any $f\in C(X)$, we define $\norm{f}_{\sup}:=\sup_{x\in X}\abs{f(x)}$.
	\item We use the letters $C$, $c$ etc. to denote positive constants whose value is inessential and may change at each occurrence. 
\item For two extended real numbers $A,B\in\mathbb{R}\cup\{-\infty,\infty\}$, let $A\wedge B:=\min(A,B)$ and $A\vee B:=\max(A,B)$.

\item For real-valued quantities $f$ and $g$, if there exists an implicit constant $C\geq1$ that depends on inessential parameters such that $f\leq Cg$ then we write $f\lesssim g$.\end{enumerate}
\end{notation}

\subsection{Preliminaries}\label{s.pre}

In this section, we recall the assumptions appearing in Theorem \ref{t.main}. We say that a triple $(X,d,m)$ is a \emph{metric measure space}, if $(X,d)$ is a locally compact separable metric space, $m$ is a nontrivial locally finite Borel regular measure on $(X,d)$, and $\#X\geq2$.

\begin{definition}
We say that $m$ satisfies the \emph{volume doubling property} \ref{VD}, or $m$ is a \emph{doubling measure} on $(X,d)$, or $(X,d,m)$ is \emph{volume doubling} \ref{VD}, if there exists $C_{\mathrm{VD}} \in (1,\infty)$ such that
	\begin{equation} \label{VD} \tag*{$\operatorname{VD}$}
	0 < m(B(x,2r)) \le C_{\mathrm{VD}} m(B(x,r)) <\infty, \quad \mbox{for all $(x,r)\in  X \times(0,\infty)$.}
	\end{equation}
\end{definition}
\begin{definition}
Let $( X  , d)$ be a metric space with $\#X\geq2$. Let $E\subset X$ be a non-empty subset. A function $u:E\to \bR$ is \emph{Lipschitz} on $E$, if there exists $L\in(0,\infty)$ such that \begin{equation}
		\abs{u(x)-u(y)}\leq L\cdot d(x,y),\ \text{ for all }(x,y)\in E\times E.
	\end{equation}
	We denote $\Lip(X,d)$ to be the collection of all Lipschitz functions on $X$. For $g\in\Lip(X,d)$, $\mathrm{LIP}[g]$ is the smallest constant $L\in[0,\infty)$ such that $\abs{g(x)-g(y)}\leq L d({x,y})$ for all $x,y\in X$. We say a function $u:X\to\bR$ is \emph{locally Lipschitz}, if for every $x\in X$, there is an open set $G\subset X$ containing $x$ such that $\restr{u}{G}$ is Lipschitz on $G$. The collection of all locally Lipschitz functions is denoted by $\Lip_{\loc}(X,d)$.
\end{definition}

The following \emph{chain condition} is taken from \cite[Definition 2.10]{KM20}.
\begin{definition} Let $(X,d)$ be a metric space with $\#X\geq2$. A sequence  $\{x_j\}_{j=0}^N$  of points in $X$ is an \emph{$\epsilon$-chain} between points $x,y \in X$ if $x_0=x$, $x_N=y$ and $d(x_j,x_{j+1}) < \epsilon$ for all $j \in\{0,1,\ldots,N-1\}$. For any $\epsilon>0$ and $x,y \in X$, define
 \[d_\epsilon(x,y) = \inf\Sett{\sum_{j=0}^{N-1} d(x_{j},x_{j+1})}{N\in\bN,\ \{x_j\}_{j=0}^N \text{ is an $\epsilon$-chain between $x$ and $y$}}.\]
The metric space $(X, d )$ is said to satisfy the \emph{chain condition}, if there exists $C \in[1,\infty)$ such that
\begin{equation}
	d_\epsilon(x,y)\leq C d(x,y),\ \text{for all }(\epsilon,x,y)\in(0,\infty)\times X\times X.
\end{equation}
\end{definition}

We follow \cite[Sections 6.2 and 5.2]{HKST15} for the following definitions of upper gradients and modulus of curve families. We also refer to \cite[Section 5.1]{HKST15} for the definition of line integrals along curves.

 \begin{definition}Let $( X  , d, m)$ be a metric measure space.
	\begin{enumerate}[label=\textup{({\arabic*})},align=right,leftmargin=*,topsep=5pt,parsep=0pt,itemsep=2pt]
	\item A Borel function $\rho:  X   \rightarrow[0, \infty]$ is said to be an \emph{upper gradient} of a map $u:  X   \rightarrow \mathbb{R}$ if
\begin{equation}\label{e.upgrad}
\abs{u(\gamma(a))- u(\gamma(b))} \leq \int_\gamma \rho \dif s \text{ for every rectifiable curve $\gamma:[a, b] \rightarrow  X  $.}
\end{equation}
\item Let $\Gamma$ be a family of curves in $ X  $. We say a Borel function $\rho:  X   \rightarrow[0, \infty]$ is an \emph{admissible density} for $\Gamma$, if \[  \begin{minipage}{270pt}
 $\displaystyle{\int_\gamma \rho \dif s \geq 1}$ for every locally rectifiable curve $\gamma \in \Gamma$.
\end{minipage}\] 
\item Let $p\in[1,\infty)$. The \emph{$p$-modulus of $\Gamma$} is defined as
\[\Mod_{p}({\Gamma}):=\inf\left\{\int_ X   \rho^p \dif m\ \Bigg|\  \text{$\rho: X  \to[0,\infty)$ is admissible} \right\}\in [0, \infty].\]
If $E\subset X$ and $F\subset X$, we define $\Mod_{p}(E,F)$ to be the $p$-modulus of all curves joining the sets $E$ and $F$ in $ X  $.
\end{enumerate}
\end{definition}

The following definition of \emph{regular local $p$-Dirichlet space} is given in \cite[Definition 2.2]{EBM25} and \cite[Definition 2.1]{EB26}.
\begin{definition}
	Let $p\in[1,\infty)$. We call $(X,d,m,\sE_{p},\sF_{p},\Gamma_{p})$ a \emph{regular local $p$-Dirichlet space}, if the following properties hold:\begin{enumerate}[label=\textup{({\arabic*})},align=right,leftmargin=*,topsep=5pt,parsep=0pt,itemsep=2pt]
	\item $(X,d)$ is a locally compact, separable metric space with a Radon measure $m$ that has full support, and $\# X \geq 2$.
	\item The space $\sF_{p}$ is a subspace of $L^{p}(X,m)$ and $(\sF_{p},\sE_{p,1}^{1/p})$ is a Banach space where $\sE_{p,1}(\cdot):=\sE_{p}(\cdot)+\norm{\cdot}^{p}_{L^{p}(X,m)}$.
	\item For each $f\in\sF_{p}$, there exists a finite non-negative Borel measure $\carre{f}$ on $X$, called the \emph{energy measure} of $f$, such that \begin{equation}
		\carre{f}(X)=\sE_{p}(f) \text{ and }\carre{\lambda f}=\abs{\lambda}^{p}\carre{f},\ \text{for all }\lambda\in\bR.
	\end{equation}
	\item For every $f,g\in\sF_{p}$ and every Borel set $A\subset X$,\begin{equation}
		\carre{f+g}(A)^{1/p}\leq\carre{f}(A)^{1/p}+\carre{g}(A)^{1/p}.
	\end{equation}
	\item For every $f\in\sF_{p}$ and every $g\in\mathrm{Lip}(\bR)$ with $g(0)=0$, we have \begin{equation}
		g\circ f\in\sF_{p}\text{ and }\dif\carre{g\circ f}\leq\mathrm{LIP}[g]^{p}\dif\carre{f}.
	\end{equation}
	\item For every $f\in\sF_{p}$ and any open set $A\subset X$, if $\restr{f}{A}=C$ $m$-a.e. for some constant $C\in\bR$, then $\carre{f}(A)=0$.
	\item For every $f\in L^{p}(X,m)$ and any sequence of functions $(f_{j})_{j\in\bN}$ in $\sF_{p}$ such that $f_{j}\to f$ in $L^{p}(X,m)$ and $\sup_{j\in\bN}\sE_{p}(f_{j})<\infty$, then \begin{equation}
		f\in\sF_{p}\text{ and }\carre{f}(A)\leq\liminf_{j\to\infty}\carre{f_{j}}(A)\ \text{for every Borel set }A\subset X.
	\end{equation}
	\item The set $C_{c}(X)\cap\sF_{p}$ is dense in $(\sF_{p},\sE_{p,1}^{1/p})$ and in $(C_{c}(X),\norm{\cdot}_{\sup})$.
\end{enumerate}

\end{definition}

We consider a continuous increasing bijection $\Psi:(0, \infty) \rightarrow(0, \infty)$ such that
\begin{equation}\label{e.scale}
	C_{\Psi}^{-1}\left(\frac{R}{r}\right)^{\beta_1} \leq \frac{\Psi(R)}{\Psi(r)} \leq C_{\Psi}\left(\frac{R}{r}\right)^{\beta_2},\ 0<r \leq R<\infty,
\end{equation}
where $C_{\Psi} \in[1, \infty)$ and $\beta_1, \beta_2 \in(0, \infty)$ are universal constants. Such a function $\Psi$ is called \emph{scale function}.

\begin{definition}[\hypertarget{PI}{$\text{Poincar\'{e} inequality}$}]
Let $(X,d,m,\sE_{p},\sF_{p},\Gamma_{p})$ be a regular local $p$-Dirichlet space.
	We say that the \emph{Poincar\'e inequality} \ref{PI} holds if there exist $C_{\mathrm{PI}} \in (0, \infty)$ and $A_{\mathrm{PI}} \in[1, \infty)$ such that for any $x \in X$, any $r \in(0, \infty)$ and any $u \in \mathcal{F}_p$,
\begin{equation}\label{PI}\tag*{$\operatorname{PI}_{p}(\Psi)$}
\int_{B(x, r)}\left|u-u_{B(x, r)}\right|^p \dif m \leq C_{\mathrm{PI}} \Psi(r)\Gamma_p\langle u\rangle(B\left(x, A_{\mathrm{PI}} r\right)),
\end{equation}
where $u_{B(x, r)}:=\fint_{B(x, r)} u \dif m={m(B(x,r))}^{-1}\int_{B(x, r)} u \dif m$. 

We say that \hyperlink{PI}{$\mathrm{PI}_{p}(\beta)$} holds if $\hyperlink{PI}{\mathrm{PI}_{p}(\Psi)}$ holds with $\Psi(r)=r^{\beta}$, $r\in[0,\infty)$.
\end{definition}
\begin{definition}
	Let $(X,d,m,\sE_{p},\sF_{p},\Gamma_{p})$ be a regular local $p$-Dirichlet space. For any two disjoint non-empty Borel sets $A$ and $B$ in $X$, we define \begin{equation}
	\operatorname{Cap}_{p}(A,B):=\inf\Sett{\sE_{p}(f)}{f\in\sF_{p}\text{ such that }\restr{{f}}{A}=0\ m\text{-a.e.}\text{ and }\restr{{f}}{B}=1\ m\text{-a.e.}}.
\end{equation}
We say that the \hypertarget{cap<}{$\text{\emph{capacity upper bound}}$} \hyperlink{cap<}{$\operatorname{Cap}_{p}(\Psi)_{\leq}$} holds if there exist $C \in(0,\infty)$ and $A \in (1,\infty)$ such that 
\begin{equation}
	\operatorname{Cap}_{p}(B(x,r),X\setminus B(x,Ar))\leq C\frac{m(B(x,r))}{\Psi(r)},\ \text{ for any }(x, r)\in X\times (0,\infty).
\end{equation}
\end{definition}
We say that \hyperlink{cap<}{$\mathrm{Cap}_{p}(\beta)_{\leq}$} holds if $\hyperlink{cap<}{\mathrm{Cap}_{p}(\Psi)_{\leq}}$ holds with $\Psi(r)=r^{\beta}$, $r\in[0,\infty)$.
\section{Proof of Theorem \ref{t.main}}\label{s.pof}
We recall the notion of upper gradient in metric measure space.
\begin{definition}Let $( X  , d, m)$ be a metric measure space.
\begin{enumerate}[label=\textup{({\arabic*})},align=right,leftmargin=*,topsep=5pt,parsep=0pt,itemsep=2pt]
\item \textup{(Cf. {\cite[Section 9.2]{ACDM15}})} A Borel function $\rho:  X   \rightarrow[0, \infty]$ is said to be a \emph{$p$-upper gradient} of a function $u:  X   \rightarrow \mathbb{R}$ if $\int_{ X  }\rho^{p}\dif m<\infty$, and there exists a function $\wt{u}$ on $ X  $ and a curve family $\Gamma$ such that $\Mod_{p}(\Gamma)=0$, $u=\wt{u}$ $m$-a.e. in $ X  $, and \begin{equation}
	\abs{\wt{u}(\gamma(a))- \wt{u}(\gamma(b))} \leq \int_\gamma \rho \dif s<\infty
\end{equation}
for every rectifiable curve $\gamma:[a, b] \rightarrow  X$ that is not in $\Gamma$. 

If the $p$-upper gradient for $u$ exists, we denote $\abs{\nabla u}_{S,p}$ by the $p$-upper gradient of $u$ with minimal $L^{p}( X  ,m)$ norm, called \emph{minimal $p$-upper gradient}, which exists by the argument in \cite[p.~54]{ACDM15}
\item For any Borel measurable real valued function $u:X\to \bR$ such that $\int_{X}\abs{u}^{p}\dif m<\infty$, we define a version of its \emph{Cheeger energy} by \begin{equation}
	\mathrm{Ch}_{p}(u):=\begin{dcases}	\int_{ X  }\abs{\nabla u}_{S,p}^{p}\dif m&\quad \text{if the $p$-upper gradient for $u$ exists,}\\
	\qquad+\infty &\quad\text{otherwise}.
\end{dcases}	
\end{equation}

\item  Let $E$ and $F$ be subsets of $ X  $. Define the following two quantities
	\[\operatorname{cap}_{p}(E, F):= \inf\Biggl\{\int_ X   \rho^p \dif m\Biggm|
 \begin{minipage}{250pt}
{$\rho$ is an upper gradient of real-valued function $u$ on $ X  $ such that $\restr{u}{E} \leq 0$ and $\restr{u}{F}  \geq 1$}
\end{minipage}
\Biggr\},\]

\[\wh{\operatorname{cap}}_p(E, F):= \inf\Biggl\{\int_ X   \rho^p \dif m\Biggm|
 \begin{minipage}{250pt}
{$\rho$ is a $p$-upper gradient of real-valued function $u$ on $ X  $ such that $\restr{u}{E} \leq 0$ and $\restr{u}{F}  \geq 1$}
\end{minipage}
\Biggr\}.\]
 \end{enumerate}
\end{definition}
\begin{remark}
	Since the minimal $p$-upper gradient minimizes the $L^{p}$ norm, we have the following upper bound on $\wh{\operatorname{cap}}_p(E, F)$:\begin{equation}\label{e.cap<Ch}
		\wh{\operatorname{cap}}_p(E, F)\leq \inf\Biggl\{\mathrm{Ch}_{p}(u)\Biggm|
\begin{minipage}{250pt}
{ $u:X\to \bR$ is a Borel measurable function such that $\int_{X}\abs{u}^{p}\dif m<\infty$, $\restr{u}{E} \leq 0$ and $\restr{u}{F}  \geq 1$}
\end{minipage}\Biggr\}.
	\end{equation}
\end{remark}

\begin{lemma}\label{l.Mod=cap}
	Let $E$ and $F$ be subsets of a metric measure space $ X  $.
	Then \begin{equation}
		\wh{\operatorname{cap}}_p(E, F)=\operatorname{cap}_{p}(E, F)=\Mod_{p}(E,F).\label{e.cap1=2}
	\end{equation}
	\end{lemma}
\begin{proof}
We first prove that $\wh{\operatorname{cap}}_{p}(E, F)\leq \operatorname{cap}_{p}(E, F)$. If $\operatorname{cap}_{p}(E, F)=\infty$, then there is nothing to prove. If $\operatorname{cap}_{p}(E, F)<\infty$, then for any $\epsilon\in(0,\infty)$, there exists a real-valued function  $u$ on $ X $ such that $\restr{u}{E} \leq 0$, $\restr{u}{F}\geq 1$; and there exists an upper gradient $\rho$ of $u$ such that $\int_{ X  }\rho^{p}\dif m<\operatorname{cap}_{p}(E, F)+\epsilon$. Evidently, $\rho$ is also a $p$-upper gradient of $u$. Therefore $\wh{\operatorname{cap}}_{p}(E, F)\leq \int_{ X  }\rho^{p}\dif m<\operatorname{cap}_{p}(E, F)+\epsilon$. Let $\epsilon\downarrow0$, we obtain $\wh{\operatorname{cap}}_{p}(E, F)\leq \operatorname{cap}_{p}(E, F)$.

We next prove that $\wh{\operatorname{cap}}_{p}(E, F)\geq \operatorname{cap}_{p}(E, F)$. Assume that $\wh{\operatorname{cap}}_{p}(E, F)<\infty$, then for any $\epsilon\in(0,\infty)$, there exists a real-valued function $u$ on $ X $ such that $\restr{u}{E} \leq 0$, $\restr{u}{F}\geq 1$; and there exists a $p$-upper gradient $\rho$ of $u$ such that $\int_{ X  }\rho^{p}\dif m<\wh{\operatorname{cap}}_{p}(E, F)+\epsilon$. Let $\abs{\nabla u}_{S,p}$ be the minimal $p$-upper gradient of $u$, then by definition \begin{equation}\label{e.mdcpa}
	\int_{X}\abs{\nabla u}_{S,p}^{p}\dif m\leq \int_{ X  }\rho^{p}\dif m<\wh{\operatorname{cap}}_{p}(E, F)+\epsilon,
\end{equation}and $\abs{\nabla u}_{S,p}$ is a \emph{$p$-weak upper gradient} for $\wt{u}$ (see \cite[Definition at p.~152]{HKST15} for definition), where $\wt{u}$ is the function in the definition of $p$-upper gradient such that $u=\wt{u}$ $m$-a.e.. Let $\rho_{u}$ be the minimal $p$-weak upper gradient of $u$ in \cite[p.~161]{HKST15}, then by \cite[Proposition 6.3.22]{HKST15}, we have $\rho_{u}=\abs{\nabla u}_{S,p}$ $m$-a.e.. So \begin{equation}\label{e.mdcpb}
	\int_{ X  }\rho_{u}^{p}\dif m=\int_{X}\abs{\nabla u}_{S,p}^{p}\dif m\overset{\eqref{e.mdcpa}}{<}\wh{\operatorname{cap}}_{p}(E, F)+\epsilon.
\end{equation} By \cite[Lemma 6.2.2]{HKST15}, there exist upper gradients $\rho_{k}$ of $u$ such that \begin{equation}\label{e.mdcpc}
	\lim_{k\to\infty}\int_{ X  }\rho_{k}^{p}\dif m=\int_{ X  }\rho^{p}_{u}\dif m.
\end{equation} By \eqref{e.mdcpb} and \eqref{e.mdcpc}, we have \begin{equation}
	\operatorname{cap}_{p}(E, F)\leq \lim_{k\to\infty}\int_{ X  }\rho_{k}^{p}\dif m<\wh{\operatorname{cap}}_{p}(E, F)+\epsilon. 
\end{equation} Letting $\epsilon\downarrow0$, we obtain $\wh{\operatorname{cap}}_{p}(E, F)\geq \operatorname{cap}_{p}(E, F)$. Combining this with the reverse inequality gives the first equality in \eqref{e.cap1=2}. The second equality $\Mod_{p}(E,F)=\operatorname{cap}_{p}(E, F)$ is proved in \cite[Theorem 7.31]{Hei01}.
\end{proof}

The following lemma provides a convenient decomposition of a metric doubling space. Recall that a metric space $(X,d)$ is called \emph{metric doubling} if there exists $N\in\bN$ such that every ball in $(X,d)$ of radius $R>0$ can be covered by at most $N$ balls of radius $R/2$. By \cite[p.~102]{HKST15}, if a metric measure space $(X,d,m)$ is \ref{VD}, then the metric space $(X,d)$ is metric doubling.

\begin{lemma}[{See \cite[Lemma 37]{ACDM15}}]\label{l.cov}
	Let $(X,d)$ be a metric doubling and separable metric space. Then for any $\delta>0$, there exist $\ell_{\delta}\in\bN\cup\{\infty\}$ and pairs $\{(A_{i}^{\delta},z_{i}^{\delta})\}_{0\leq i<\ell_{\delta}}$, where $A_{i}^{\delta}\subset  X $ are Borel sets and $z_{i}^{\delta}\in  X  $ satisfying \begin{enumerate}[label=\textup{({\arabic*})},align=right,leftmargin=*,topsep=5pt,parsep=0pt,itemsep=2pt]
	\item\label{it.cov1} The sets $\{A_{i}^{\delta}\}_{0\leq i<\ell_{\delta}}$ are a partition of $ X  $, that is they are pairwise disjoint, $ X  =\bigcup_{0\leq i<\ell_{\delta}}A_{i}^{\delta}$. Moreover, $d(z_{i}^{\delta},z_{j}^{\delta})>\delta$ whenever $i\neq j$.
	\item\label{it.cov2} $A_{i}^{\delta}$ are comparable to balls centred at $z_{i}^{\delta}$, namely \begin{equation}\label{e.nested}
		B\big(z_{i}^{\delta},\frac{\delta}{3}\big)\subset A_{i}^{\delta}\subset B\big(z_{i}^{\delta},\frac{5}{4}\delta\big),\ \text{for all }0\leq i<\ell_{\delta}.
	\end{equation}
	\item\label{it.cov3} For any $K\in(0,\infty)$, there exists a number $N(K)\in\bN$ depending only on $K$, such that \begin{equation}\label{e.bolp}
		\sum_{0\leq i<\ell_{\delta}}\one_{B(z_{i}^{\delta},K\delta)}\leq N(K)<\infty.
	\end{equation}
\end{enumerate}
We say that $A_{i}^{\delta}$ is a \emph{neighbour} of $A_{j}^{\delta}$, and we denote by $A_{i}^{\delta}\sim A_{j}^{\delta}$, if their distance is less than $\delta$.  In particular $A_{i}^{\delta}\sim A_{j}^{\delta}$ implies that $d(z_{i}^{\delta}, z_{j}^{\delta})<4\delta$.
\end{lemma}
\begin{proof}
	The assertions \ref{it.cov1} and \ref{it.cov2} are proved in \cite[Lemma 37]{ACDM15}. We now prove \ref{it.cov3}. Let $y\in X$. Define \begin{equation}
		I_{y}:=\Sett{z\in \{z_{i}^{\delta}\}_{0\leq i<\ell_{\delta}}}{y\in B(z_{i}^{\delta},K\delta)}.
	\end{equation}
	 Note that $I_{y}\subset B(y, K\delta)$, and for every $z,w \in I_{y}$ with $z\neq w$, we have $d(z,w)\geq \delta$. By \cite[Exercise 10.17]{Hei01}, we know that there exists $N(K)\in\bN$ depending only on $K$ such that $\# I_{y}\leq N(K)$. By noting that $\sum_{0\leq i<\ell_{\delta}}\one_{B(z_{i}^{\delta},K\delta)}(y)=\# I_{y}$, we complete the proof of \eqref{e.bolp}. 
\end{proof}
Let $u\in L^{p}( X  ,m)$. We define the average $u_{\delta,i}$ of $u$ in each cell of the partition by \begin{equation}
	u_{\delta,i}:=\fint_{A_{i}^{\delta}}u\dif m=\frac{1}{m(A_{i}^{\delta})}\int_{A_{i}^{\delta}}u\dif m.
\end{equation} We then define a function $\abs{\sD_{\delta}u}$ by \begin{equation}
	\abs{\sD_{\delta}u}^{p}(x):=\sum_{0\leq i<\ell_{\delta}}\frac{1}{\delta^{p}}\one_{A_{i}^{\delta}}(x)\Big(\sum_{j: A_{i}^{\delta}\sim A_{j}^{\delta}}\abs{u_{\delta,i}-u_{\delta,j}}^{p}\Big), \ \text{for all } x\in  X  .
\end{equation} 
We also define an approximate gradient by \begin{align}
	\rF_{\delta}(u)&:=\int_{ X  }\abs{\sD_{\delta}u}^{p}(x)\dif m(x)=\sum_{0\leq i<\ell_{\delta}}\frac{1}{\delta^{p}}m(A_{i}^{\delta})\Big(\sum_{j: A_{i}^{\delta}\sim A_{j}^{\delta}}\abs{u_{\delta,i}-u_{\delta,j}}^{p}\Big).
\end{align}

We recall the definition of \emph{$\Gamma$-convergence}. Let $(Y,d_Y)$
be a metric space, and let $F_j:Y\to[-\infty,\infty]$, $j\in\bN$, be
a sequence of functions. We say that $(F_j)_{j\in\bN}$
\emph{$\Gamma$-converges} to a function
$F:Y\to[-\infty,\infty]$ if the following conditions hold:
\begin{enumerate}[label=\textup{(\roman*)},align=right,leftmargin=*,topsep=5pt,parsep=0pt,itemsep=2pt]
    \item For every $y\in Y$ and every sequence $(y_j)_{j\in\bN}
    \subset Y$ converging to $y$,
    \begin{equation}\label{e.gamc1}
        F(y)\leq\liminf_{j\to\infty}F_j(y_j).
    \end{equation}

    \item For every $y\in Y$, there exists a sequence
    $(y_j)_{j\in\bN}\subset Y$ converging to $y$ such that
    \begin{equation}\label{e.gamc2}
        F(y)\geq\limsup_{j\to\infty}F_j(y_j).
    \end{equation}
    Such a sequence is called a \emph{recovery sequence} for $y$.
\end{enumerate}

\begin{lemma}[{Cf. \cite[Theorem 40]{ACDM15}}]
	Let $\{\delta_{k}\}_{k\in\bN}$ be any sequence such that $\delta_{k}\downarrow 0$ as $k\uparrow\infty$. Let $\rF$ be a $\Gamma$-limit of $\rF_{\delta_{k}}$ as $k\to\infty$. Then there exists $\eta\in[1,\infty)$ depending only on $C_{\mathrm{VD}}$ and $p$ such that \begin{equation}\label{e.Dis=Cont}
		\eta^{-1} \mathrm{Ch}_{p}(u)\leq \rF(u)\leq \eta \mathrm{Ch}_{p}(u),\ \text{ for all } u\in L^{p}( X  ,m).
	\end{equation}
\end{lemma}
\begin{proof}
	In \cite[Theorem 40]{ACDM15}, the inequalities \eqref{e.Dis=Cont} are proved for the Cheeger energy $\mathrm{Ch}_{p}$ defined using \emph{minimal $p$-relaxed slope} (cf. \cite[Definition 15]{ACDM15}) and \emph{minimal $p$-weak upper gradient} (cf. \cite[Definition 21]{ACDM15} and \cite[p.~161]{HKST15}). Nevertheless, the minimal $p$-relaxed slope coincides with the minimal $p$-upper gradient by \cite[(9.1), (9.4), (9.5) and Theorem 35; argument at p.~54]{ACDM15}.
\end{proof}
The following elementary lemma, which compares two standard forms of oscillation, is well known.
\begin{lemma}
	For any $p\in[1,\infty)$, any $f\in L^{p}( X  ,m)$ and any ball $B=B(x,r)$, we have \begin{align}\label{e.PIcomp}
		 2^{-p}\fint_{B}\abs{f-f_{B}}^{p}\dif m&\leq \fint_{B}\fint_{B}\abs{f(x)-f(y)}^{p}\dif m(y)\dif m(x)\leq \fint_{B}\abs{f-f_{B}}^{p}\dif m.
	\end{align}
\end{lemma}
\begin{proof}
By triangle inequality,
	\begin{align}
		&\phantom{\ \leq}\int_{B}\int_{B}\abs{f(x)-f(y)}^{p}\dif m(y)\dif m(x)\\
		&\leq 2^{p-1}\int_{B}\int_{B}\abs{f(x)-f_{B}}^{p}+\abs{f(y)-f_{B}}^{p}\dif m(y)\dif m(x)\\
		&=2^{p}m(B)\int_{B}\abs{f(x)-f_{B}}^{p}\dif m(x).
	\end{align}
	Dividing $m(B)^{2}$ on both sides, we obtain the second inequality of \eqref{e.PIcomp}.
	On the other hand, \begin{align}
		&\phantom{\ \leq}\int_{B}\abs{f(x)-f_{B}}^{p}\dif m(x)=\int_{B}\abs{\fint_{B}(f(x)-f(y))\dif m(y)}^{p}\dif m(x)\\
		&\leq \int_{B}\fint_{B}\abs{f(x)-f(y)}^{p}\dif m(y)\dif m(x) \quad \text{(by H\"older's inequality)}.
	\end{align}
	Dividing $m(B)$ on both sides, we obtain the first inequality of \eqref{e.PIcomp}.
\end{proof}
We are ready to prove Theorem \ref{t.main}-\ref{it.main1}.
\begin{proof}[Proof of Theorem \ref{t.main}-\ref{it.main1}]
	Let $\delta\in(0,\infty)$. Let $\{(A_{i}^{\delta},z_{i}^{\delta})\}_{0\leq i<\ell_{\delta}}$ be the pairs of Borel sets and points in Lemma \ref{l.cov}. For any $i$ and $j$ such that $A_{i}^{\delta}\sim A_{j}^{\delta}$, we have by Lemma \ref{l.cov} that \begin{equation}\label{e.union}
		A_{i}^{\delta}\cup A_{j}^{\delta}\subset B(z_{i}^{\delta},5\delta).
	\end{equation} By the volume doubling property \ref{VD}, there exists a constant $C_{1}\in(1,\infty)$ such that \begin{equation}\label{e.volm1}
		C_{1}^{-1}m(B\big(z_{i}^{\delta},5\delta \big))\overset{\eqref{VD}}{\leq }m(B\big(z_{i}^{\delta},\frac{\delta}{3}\big))\overset{\eqref{e.nested}}{\leq} m(A_{i}^\delta)\overset{\eqref{e.nested}}{\leq} m(B\big(z_{i}^{\delta},\frac{5}{4}\delta\big))\overset{\eqref{VD}}{\leq }C_1m(B\big(z_{i}^{\delta},\frac{\delta}{3}\big))
	\end{equation}
	and \begin{equation}\label{e.volm2}
	C_{1}^{-1}m(B\big(z_{i}^{\delta},\frac{\delta}{3}\big)) \overset{\eqref{e.nested},\eqref{e.union}}{\leq} C_{1}^{-1}m(B\big(z_{j}^{\delta},5\delta \big))\overset{\eqref{VD}}{\leq }	m(B\big(z_{j}^{\delta},\frac{\delta}{3}\big))\overset{\eqref{e.nested}}{\leq} m(A_{j}^\delta).
	\end{equation}
	Then \begin{align}
		\abs{u_{\delta,i}-u_{\delta,j}}^{p}&=\abs{\fint_{A_{i}^{\delta}}u\dif m-\fint_{A_{j}^{\delta}}u\dif m}^{p}\\
		&\leq\fint_{A_{i}^{\delta}}\fint_{A_{j}^{\delta}}\abs{u(x)-u(y)}^{p}\dif m(x)\dif m(y) \ \ \text{(by H\"older's inequality)}\\
		&{\lesssim} \fint_{B(z_{i}^{\delta},5\delta)}\fint_{B(z_{i}^{\delta},5\delta)}\abs{u(x)-u(y)}^{p}\dif m(x)\dif m(y)\ \ \text{(by \eqref{e.union}, \eqref{e.volm1} and \eqref{e.volm2})}\\
		&\overset{\eqref{e.PIcomp}}{\lesssim} \fint_{B(z_{i}^{\delta},5\delta)}\abs{u-u_{B(z_{i}^{\delta},5\delta)}}^{p}\dif m.\label{e.vmpi}
	\end{align}
	Therefore for any $u\in\sF_{p}$, \begin{align}
		\rF_{\delta}(u)&=\sum_{0\leq i<\ell_{\delta}}\frac{1}{\delta^{p}}m(A_{i}^{\delta})\Big(\sum_{j: A_{i}^{\delta}\sim A_{j}^{\delta}}\abs{u_{\delta,i}-u_{\delta,j}}^{p}\Big)\\
		&\overset{\eqref{e.volm1},\eqref{e.vmpi}}{\lesssim}\sum_{0\leq i<\ell_{\delta}}\frac{1}{\delta^{p}}\int_{B(z_{i}^{\delta},5\delta)}\abs{u-u_{B(z_{i}^{\delta},5\delta)}}^{p}\dif m \overset{\hyperlink{PI}{\mathrm{PI}_{p}(\Psi)}}{\lesssim}\sum_{0\leq i<\ell_{\delta}}\frac{\Psi(5\delta)}{\delta^{p}}\int_{B(z_{i}^{\delta},5A_{\mathrm{PI}}\delta)}\dif \Gamma_{p}\langle u\rangle\\
		&\overset{\eqref{e.scale}}{\lesssim} \frac{\Psi(\delta)}{\delta^{p}}\int_{ X  }\sum_{0\leq i<\ell_{\delta}}\one_{B(z_{i}^{\delta},5A_{\mathrm{PI}}\delta)}\dif \Gamma_{p}\langle u\rangle \overset{\eqref{e.bolp}}{\leq} CN(5A_{\mathrm{PI}})\cdot\frac{\Psi(\delta)}{\delta^{p}}\sE_{p}(u).\label{e.F<Energy}
	\end{align}
	
	By \eqref{e.ma1}, there exists a positive sequence $\{\delta_{k}\}_{k\in\bN}\subset (0,\infty)$ such that $\delta_{k}\downarrow0$ as $k\uparrow\infty$, and $\lim_{k\to\infty}\delta_{k}^{-p}\Psi(\delta_{k})=0$. Along this sequence, since $u$ itself converges to $u$ in $L^{p}(X,m)$, we have \begin{align}
		\mathrm{Ch}_{p}(u)&\overset{\eqref{e.Dis=Cont}}{\leq} \eta \rF(u)\\
		& \overset{\eqref{e.gamc1}}{\leq}\eta\liminf_{k\to\infty}\rF_{\delta_{k}}(u)\ \text{ (by definition of $\Gamma$-convergence) } \\
		&\overset{\eqref{e.F<Energy}}{\leq}C\eta \liminf_{k\to\infty}\delta_{k}^{-p}\Psi(\delta_{k})\sE_{p}(u)=0\ \text{ (since $\sE_{p}(u)<\infty$)}.
	\end{align}
	Therefore \begin{equation}\label{e.Ch=0}
		\text{$\mathrm{Ch}_{p}(u)=0$ for all $u\in\sF_{p}$.}
	\end{equation} 	
	
	We now prove \ref{VMp}. Let $K$ be a compact subset of $ X  $ and $G$ be a relatively compact open subset of $ X  $ with $K\Subset G$. Then by \cite[Proposition 3.28]{KS26} (see also \cite[Proposition A.4]{CY26}), there exists $\varphi\in\sF_{p}$ such that $\restr{\varphi}{K}=1$ $m$-a.e. and $\restr{\varphi}{X  \setminus G}=0$ $m$-a.e.. By Lemma \ref{l.Mod=cap}, we have \begin{equation}
		0\leq \Mod_{p}(K, X\setminus G)\overset{\eqref{e.cap1=2}}{=}\wh{\operatorname{cap}}_{p}(K, X\setminus G)\overset{\eqref{e.cap<Ch}}{\leq} \mathrm{Ch}_{p}(\varphi)\overset{\eqref{e.Ch=0}}{=}0.
	\end{equation} Taking a countable base of $ X  $ and use the sub-additivity of $\Mod_{p}$ (see \cite[Eq.~(5.2.6)]{HKST15}), we  obtain \ref{VMp}.
	\end{proof}

We now prove Theorem \ref{t.main}-\ref{it.main2}. We first record the following lemma, which follows from the McShane--Whitney extension theorem \cite[p.~99]{HKST15}, together with the fact that a locally Lipschitz function is Lipschitz on compact subsets.
\begin{lemma}
	Let $(X,d)$ be a metric space. Then for every $u\in \Lip_{\loc}(X,d)$ and every compact set $K\subset X$, there exists $v\in \Lip(X,d)$ such that $\restr{v}{K}=\restr{u}{K}$.
\end{lemma}
We record the following definition of pointwise upper Lipschitz-constant and the restricted capacities.
\begin{definition}
	Let $(X,d,m)$ be a metric measure space.
	\begin{enumerate}[label=\textup{({\arabic*})},align=right,leftmargin=*,topsep=5pt,parsep=0pt,itemsep=2pt]
	\item For every $u: X\to\bR$, we define the \emph{pointwise upper Lipschitz-constant function} of $u$ by \begin{equation}
		\lip{d}{u}(x):=\limsup_{r\downarrow 0}\sup_{y\in B(x,r)}\frac{\abs{u(x)-u(y)}}{d(x,y)},\ x\in X.
	\end{equation}
	\item For every two disjoint closed sets $E$ and $F$ in $X$, we define \begin{align}
		&\phantom{:=}\operatorname{cap}^{\mathrm{(lip,lip)}}_{p}(E,F)\\
		&:= \inf\Biggl\{\int_{X}g^{p}\dif m   \Biggm|
  \begin{minipage}{250pt}
$g\in \Lip_{\loc}(X,d)$ is an upper gradient for some $u\in \Lip_{\loc}(X,d)$ such that $\restr{u}{E}=0 $ and $\restr{u}{F}=1$ \end{minipage}
\Biggr\}
	\end{align}
\end{enumerate}
\begin{remark}
	\begin{enumerate}[label=\textup{({\arabic*})},align=left,leftmargin=*,topsep=5pt,parsep=0pt,itemsep=2pt]
	\item By \cite[Lemma 4.1.2]{Kei04}, for any $u\in \Lip_{\loc}(X,d)$, the function $\lip{d}{u}$ is Borel measurable.
	\item Suppose $(X,d)$ is geodesic, suppose $u\in \Lip_{\loc}(X,d)$ and suppose $g\in \Lip_{\loc}(X,d)$ is an upper gradient for $u$. Then as in \cite[Proof of Lemma~11]{Kei03},\begin{equation}\label{e.lip<g}
		\lip{d}{u}(x)\leq g(x), \ \text{ for all } x\in X.
	\end{equation} 
	In fact, let $x,y\in X$ and let $\gamma_{x,y}$ be a geodesic connecting $x$ and $y$. Then $\mathrm{Length}(\gamma_{x,y})=d(x,y)$. Since $g$ is an upper gradient, we have \begin{equation}
		\frac{\abs{u(x)-u(y)}}{d(x,y)}\leq \frac{1}{d(x,y)}\int_{\gamma_{x,y}}g \dif s = \fint_{\gamma_{x,y}}g \dif s
	\end{equation}
	Letting $y\to x$ and using the continuity of $g$, we obtain $\lip{d}{u}(x)\leq g(x)$.
\end{enumerate}

\end{remark}
\end{definition}
\begin{definition}\label{d.floc}
Let $(X,d,m,\sE_{p},\sF_{p},\Gamma_{p})$ be a regular local $p$-Dirichlet space. Define \begin{equation}
		\sF_{p,\loc}:=
  \Biggl\{f\in L^{p}_{\mathrm{loc}}( X ,m)  \Biggm|
  \begin{minipage}{260pt}
For any relatively compact open subset $A$ of $(X,d)$,  there exists $f^{\#}\in \sF_{p}$ such that $\restr{f^{\#}}{A}=\restr{f}{A}$ $m$-a.e.
\end{minipage}
\Biggr\}
	\end{equation}

\end{definition}
\begin{proposition}\label{p.Lip<Sob}
Let $(X,d)$ be a complete metric space that satisfies the chain condition. Let $(X,d,m,\sE_{p},\sF_{p},\Gamma_{p})$ be a regular local $p$-Dirichlet space that satisfies \ref{VD}, $\hyperlink{PI}{\mathrm{PI}_{p}(\Psi)}$ and \hyperlink{cap<}{$\operatorname{Cap}_{p}(\Psi)_{\leq}$}. If the scale function $\Psi$ satisfies \begin{equation}\label{e.sup>0}
	\limsup_{r\downarrow 0}\frac{\Psi(r)}{r^{p}}>0,
\end{equation}
then the following hold:

\begin{enumerate}[label=\textup{({\arabic*})},align=right,leftmargin=*,topsep=5pt,parsep=0pt,itemsep=2pt]
	\item\label{it.lip<sob1} The set $\sF_{p,\loc}$ contains all locally Lipschitz functions, that is, $\Lip_{\loc}(X,d)\subset \sF_{p,\loc}$. Moreover, there exists $C\in[1,\infty)$ such that for every $u\in\Lip_{\loc}(X,d)$,\begin{equation}\label{e.sob1}
		\carre{u}\ll m\ \text{ and }\ \frac{\dif\carre{u}}{\dif m}(x)\leq C\cdot \lip{d}{u}(x)^{p},\ m\text{-a.e. }x\in X.
	\end{equation}
	\item\label{it.lip<sob2} The metric space $(X,d)$ supports a weak $(1,p)$-Poincar\'{e} inequality: there exist $C\in(0,\infty)$ such that for every $(x,r)\in X\times (0,\infty)$ and for every $u\in\Lip_{\loc}(X,d)$, \begin{equation}\label{e.sob2}
		\fint_{B(x,r)}\abs{u-u_{B(x,r)}}\dif m \leq C\Psi(r)^{1/p}\left( \fint_{B(x,A_{\mathrm{PI}} r)}\lip{d}{u}^{p}\dif m\right)^{1/p}.
	\end{equation}
	where $A_{\mathrm{PI}}\in(1,\infty)$ is the constant appeared in $\hyperlink{PI}{\mathrm{PI}_{p}(\Psi)}$.
	\item\label{it.lip<sob3} There exists $R_{0}\in(0,\diam(X,d))$ and $C\in(0,\infty)$ such that \begin{equation}\label{e.sob3}
		C^{-1}r^{p}\leq \Psi(r)\leq C r^{p}\ \text{ for all }r\in(0,R_{0}].
	\end{equation}
\end{enumerate}
\end{proposition}
\begin{proof}
	\begin{enumerate}[label=\textup{({\arabic*})},align=right,leftmargin=*,topsep=5pt,parsep=0pt,itemsep=2pt]
	\item[\ref{it.lip<sob1}]By \cite[Theorem 1.4]{EB26}, we know that \ref{VD}, $\hyperlink{PI}{\mathrm{PI}_{p}(\Psi)}$ and \hyperlink{cap<}{$\operatorname{Cap}_{p}(\Psi)_{\leq}$} together imply a \emph{cutoff Sobolev inequality} (see \cite[p.~4]{EB26} and \cite[p.~4]{Yan25b} for definition)x. By \ref{VD}, $\hyperlink{PI}{\mathrm{PI}_{p}(\Psi)}$ and the cutoff Sobolev inequality, we know from \cite[Theorems 2.4 and 2.5]{Yan25b} that $\Lip_{\loc}(X,d)\subset \sF_{p,\loc}$ and \eqref{e.sob1} holds.
	\item[\ref{it.lip<sob2}] By the locality of energy measures, the definition of $\sF_{p,\loc}$ in Definition \ref{d.floc}, and the precompactness of each metric ball, we know that $\hyperlink{PI}{\mathrm{PI}_{p}(\Psi)}$ actually holds for all functions in $\sF_{p,\loc}$. Now fix $u\in\Lip_{\loc}(X,d)$, we have \begin{align}
		\fint_{B(x,r)}\abs{u-u_{B(x,r)}}\dif m &\leq \left(\fint_{B(x,r)}\abs{u-u_{B(x,r)}}^{p}\dif m )\right)^{1/p} \ \text{ (Jensen's inequality)}\\
		&\leq C^{1/p}\Psi(r)^{1/p} \left(\fint_{B(x,A_{\mathrm{PI}} r)}\dif\carre{u}\right)^{1/p} \ \text{ ($\hyperlink{PI}{\mathrm{PI}_{p}(\Psi)}$ for $\sF_{p,\loc}$)}\\
		&\leq C_{1}\Psi(r)^{1/p}\left( \fint_{B(x,A_{\mathrm{PI}} r)}\lip{d}{u}^{p}\dif m\right)^{1/p} \ \text{ (by \eqref{e.sob1})}.
	\end{align}
	This proves \eqref{e.sob2}.
	\item[\ref{it.lip<sob3}] This is proved in \cite[Lemma 4.1]{Yan25b}.
\qedhere
\end{enumerate}
\end{proof}
The following lemma is adapted from \cite[Theorem 5.9]{HK98}.
\begin{lemma}
	Let $(X,d)$ be a complete metric space that satisfies the chain condition. Let $(X,d,m,\sE_{p},\sF_{p},\Gamma_{p})$ be a regular local $p$-Dirichlet space that satisfies \ref{VD}, $\hyperlink{PI}{\mathrm{PI}_{p}(\Psi)}$ and \hyperlink{cap<}{$\operatorname{Cap}_{p}(\Psi)_{\leq}$}. Suppose the scale function $\Psi$ satisfies \eqref{e.sup>0}. Let $R_{0}$ be the constant appearing in Proposition \ref{p.Lip<Sob}-\ref{it.lip<sob3}. Then there is a constant $C_{0}\in(1,\infty)$ such that, for any $(z,r)\in X\times(0,R_{0}]$, any two Borel subsets $E,F\subset B(z,r)$, and any $u\in\Lip_{\loc}(X,d)$ with $\restr{u}{E}\leq0$ and $\restr{u}{F}\geq1$, we have\begin{equation}\label{e.mod1}
	\int_{B(z,10A_{\mathrm{PI}} r)}\lip{d}{u}^{p}\dif m\geq C_{0}^{-1}\frac{m(E)\vee m(F)}{r^{p}}.
\end{equation}
\end{lemma}
\begin{proof}
	Let $(z,r)\in X\times(0,R_{0}]$ and let $E,F\subset B(z,r)$ be two Borel subsets. Fix $u\in\Lip_{\loc}(X,d)$ with $\restr{u}{E}\leq0$ and $\restr{u}{F}\geq1$. We consider the set \begin{equation}
		A:=\Sett{(x,y)\in E\times F}{\abs{u(x)-u_{B(x,r)}}\vee \abs{u(y)-u_{B(y,r)}}\leq \frac{1}{5}}.
	\end{equation}
	\begin{enumerate}[label=\textit{Case {\arabic*}.},align=right,leftmargin=*,topsep=5pt,parsep=0pt,itemsep=2pt]
	\item Suppose $A\neq \emptyset$. Let $(x_{0},y_{0})\in A$. Then \begin{align}
		1\leq \abs{u(x_{0})-u(y_{0})}&\leq \frac{1}{5}+\abs{u_{B(x_{0},r)}-u_{B(y_{0},r)}}+\frac{1}{5}\\
		&\leq \frac{2}{5}+\fint_{B(x_{0},r)}\fint_{B(y_{0},r)}\abs{u(x)-u(y)}\dif m(y)\dif m(x)\\
		&\overset{\eqref{VD}}{\leq} \frac{2}{5}+C\fint_{B(z,2r)}\fint_{B(z,2r)}\abs{u(x)-u(y)}\dif m(y)\dif m(x)\\
		&\leq \frac{2}{5}+C\fint_{B(z,2r)}\abs{u-u_{B(z,2r)}}\dif m\\
		&\overset{\eqref{e.sob2},\eqref{e.sob3}}{\leq} \frac{2}{5}+C_{1}r\left( \fint_{B(z,2A_{\mathrm{PI}} r)}\lip{d}{u}^{p}\dif m\right)^{1/p}
	\end{align}
	By rearranging terms, we obtain \eqref{e.mod1}. 
	\item Suppose $A=\emptyset$. Then for every $x\in E$,\begin{align}
		\frac{1}{5}\leq \abs{u(x)-u_{B(x,r)}}&=\sum_{j=0}^{\infty}\abs{u_{B(x,2^{-j}r)}-u_{B(x,2^{-j-1}r)}}\\
		&\leq C_{2}\sum_{j=0}^{\infty}\fint_{B(x,2^{-j}r)}\abs{u-u_{B(x,2^{-j}r)}}\dif m\\
		&\overset{\eqref{e.sob2},\eqref{e.sob3}}{\leq} C_{3}\sum_{j=0}^{\infty}(2^{-j}r)\left( \fint_{B(x,A_{\mathrm{PI}}2^{-j}r)}\lip{d}{u}^{p}\dif m\right)^{1/p}
	\end{align}
	Therefore, there exists $j_{x}\in\bN\cup\{0\}$ such that \begin{equation}\label{e.mod2}
		C_{4}(2^{-j_{x}}r)\left( \fint_{B(x,A_{\mathrm{PI}}2^{-j_{x}}r)}\lip{d}{u}^{p}\dif m\right)^{1/p}\geq 2^{-j_{x}}.
	\end{equation}
	Since $\{B(x,A_{\mathrm{PI}}2^{-j_{x}}r)\}_{x\in E}$ is a covering for $E$, by \cite[Theorem 1.2]{Hei01}, there exists an at most countable set $E_{0}\subset E$ such that \begin{equation}\label{e.mod3}
		E\subset \bigcup_{x\in E_{0}}B(x,5A_{\mathrm{PI}}2^{-j_{x}}r), \text{ and }\{B(x,A_{\mathrm{PI}}2^{-j_{x}}r)\}_{x\in E_{0}} \text{ are mutually disjoint}.
	\end{equation}
	Note that \begin{align}
		\int_{B(z,10A_{\mathrm{PI}} r)}\lip{d}{u}^{p}\dif m&\overset{\eqref{e.mod3}}{\geq} \sum_{x\in E_{0}} \int_{B(x,A_{\mathrm{PI}}2^{-j_{x}}r)}\lip{d}{u}^{p}\dif m\\
		&\overset{\eqref{e.mod2}}{\geq} \sum_{x\in E_{0}}C_{4}^{-p}r^{-p} m(B(x,A_{\mathrm{PI}}2^{-j_{x}}r))\\
		&\overset{\eqref{VD}}{\geq} \sum_{x\in E_{0}}C_{5}^{-1}r^{-p} m(B(x,5A_{\mathrm{PI}}2^{-j_{x}}r))\overset{\eqref{e.mod3}}{\geq}C_{5}^{-1}\frac{m(E)}{r^{p}}.
	\end{align}
	The same conclusion holds for $F$ in place of $E$.
	\qedhere
	\end{enumerate}
\end{proof}
\begin{proof}[Proof of Theorem \ref{t.main}-\ref{it.main2}]
	Since $(X,d)$ satisfies the chain condition, by a bi-Lipschitz transform on the metric \cite[Proposition A.1]{KM20}, we may assume that $(X,d)$ is geodesic. Let $(z,r)\in X\times(0,R_{0}]$, where $R_{0}$ is the constant appearing in Proposition \ref{p.Lip<Sob}-\ref{it.lip<sob3}. Let $E,F\subset B(z,r)$ be two disjoint closed subsets. Let $g\in \Lip_{\loc}(X,d)$ be an upper gradient for some $u\in \Lip_{\loc}(X,d)$ such that $\restr{u}{E}=0 $ and $\restr{u}{F}=1$. Then \begin{equation}
		\int_{X}g^{p}\dif m\overset{\eqref{e.lip<g}}{\geq} \int_{B(z,10A_{\mathrm{PI}} r)}\lip{d}{u}^{p}\dif m\overset{\eqref{e.mod1}}{\geq} C_{0}^{-1}\frac{m(E)\vee m(F)}{r^{p}}.
	\end{equation}
	Therefore \begin{equation}\label{e.pf22}
			\operatorname{cap}^{\mathrm{(lip,lip)}}_{p}(E,F)\geq C_{0}^{-1}\frac{m(E)\vee m(F)}{r^{p}},
	\end{equation}
	and consequently \begin{align}
		\Mod_{p}(E,F)&=\operatorname{cap}_{p}(E,F)\\
		&=\operatorname{cap}^{\mathrm{(lip,lip)}}_{p}(E,F)\ \ \text{ (by  \cite[Theorem 1.1]{EBPC24})}\\
		&\overset{\eqref{e.pf22}}{\geq} C_{0}^{-1}\frac{m(E)\vee m(F)}{r^{p}}.
	\end{align}
	In particular, by choosing suitable $E$ and $F$ with positive mass, we know that \ref{VMp} cannot hold. 
\end{proof}
\section{Application to Ahlfors-regular conformal dimension}\label{s.app}

In this section, we prove Theorem \ref{t.dic}. We will also show that every weak tangent of the generalized Sierpi\'{n}ski carpet satisfies \ref{VMp} for every $p\in[1,\infty)$, which gives an alternative proof of its non-minimality for Ahlfors-regular conformal dimension. We first recall the following definitions.
\begin{definition} 
Let $(X,d)$ be a metric space with $\# X\geq2$.
\begin{enumerate}[label=\textup{({\arabic*})},align=right,leftmargin=*,topsep=5pt,parsep=0pt,itemsep=2pt]
\item We say that a metric $\theta$ on $X$ is \emph{quasisymmetric} to $d$, if there exists a homeomorphism of $[0,\infty)$ onto itself, denoted by $\eta$, such that \begin{equation}\label{e.defqs1}
\frac{\theta(x,y)}{\theta(x,y)} \le \eta\left(\frac{d(x,y)}{d(x,z)}\right),\ \text{ for all }x,y,z \in X \text{ with } x \neq z.  
\end{equation}
\item Let $\alpha\in(0,\infty)$. Let $m$ be a Radon measure on $(X,d)$. We say that $(X,d,m)$ is \emph{Ahlfors $\alpha$-regular}, if there exists a constant $C\in(1,\infty)$ such that \begin{equation}
	C^{-1}r^{\alpha}\leq m(B(x,r))\leq Cr^{\alpha},\ \text{ for all }(x,r)\in X\times (0,\diam(X,d)).
\end{equation}
In this case, $m$ is comparable to the $\alpha$-Hausdorff measure on $(X,d)$; see \cite[Exercise 8.11]{Hei01}. The metric space $(X,d)$ is called \emph{Ahlfors $\alpha$-regular}, if there exists such a measure $m$ such that $(X,d,m)$ is {Ahlfors $\alpha$-regular}.

 It is well-known that if $(X,d)$ is {Ahlfors $\alpha$-regular}, then its Hausdorff dimension $\dim_{\mathrm{H}}(X,d)=\alpha$; see \cite[discussion above Exercise 8.11]{Hei01}.

The metric space $(X,d)$ is said to be \emph{Ahlfors regular}, if it is {Ahlfors $\alpha$-regular} for some $\alpha\in(0,\infty)$.  
\item The \emph{Ahlfors-regular conformal dimension} of $(X,d)$ is defined by \begin{equation}
\dim_{\mathrm{ARC}}(X,d):=\inf\Biggl\{\dim_{\mathrm{H}}(X,\theta)\Biggm|
 \begin{minipage}{200pt}
$\theta$ is a metric on $X$ that is quasisymmetric to $d$, and $(X,\theta)$ is Ahlfors regular
\end{minipage}
\Biggr\}.
\end{equation}
\end{enumerate}
\end{definition}
We briefly introduce the definition of \emph{Dirichlet form}. For details, see \cite[Chapter 1]{FOT11}.
\begin{definition}
Given a metric measure space $(X,d,m)$. We say that $(\mathcal{E},\mathcal{F})$ is a \emph{regular Dirichlet form} on $L^{2}(X,m)$, if it satisfies the following conditions:
\begin{enumerate}[label=\textup{({\roman*})},align=right,leftmargin=*,topsep=5pt,parsep=0pt,itemsep=2pt]
    \item $\mathcal{F}$ is a dense linear subspace of $L^{2}(X,m)$;
    \item $\mathcal{E}$ is a non-negative definite symmetric bilinear form on $\mathcal{F}\times\mathcal{F}$;
    \item $\mathcal{F}$ is a Hilbert space with inner product $\mathcal{E}_{1}:=\mathcal{E}+\langle\cdot,\cdot\rangle_{L^{2}(X,m)}$;
    \item for every $u\in \mathcal{F}$, the function $v:=(0\vee u)\wedge1\in\mathcal{F}$ and $\mathcal{E}(v, v)\leq\mathcal{E}(u, u)$;
    \item $\mathcal{F}\cap C_{c}(X)$ is dense both in $(\mathcal{F},\sqrt{\mathcal{E}_{1}})$ and in $(C_{c}(X),\lVert\cdot\rVert_{\sup})$.
\end{enumerate}If, in addition, $(\mathcal{E},\mathcal{F})$ satisfies
\begin{enumerate}[label=\textup{({\roman*})},align=right,leftmargin=*,topsep=5pt,parsep=0pt,itemsep=2pt,resume]
  \item $\mathcal{E}(f, g)=0$ for any $f, g\in\mathcal{F}$ with $\supp_{m}(f)$, $\supp_{m}(g)$ compact, and $\supp_{m}(f-a\one_{X})\cap \supp_{m}(g)=\emptyset$ for some $a\in\mathbb{R}$. Here, $\supp_{m}(f)$ denotes the support of the measure $|f|\dif m$,
\end{enumerate}
then we say $(\mathcal{E},\mathcal{F})$ is \emph{strongly local}. Let $\{P_{t}\}_{t\in(0,\infty)}$ denote the semigroup corresponding to $(\mathcal{E},\mathcal{F})$ on $L^{2}(X,m)$. We say $(\mathcal{E},\mathcal{F})$ is \emph{conservative} if $P_{t}\one_{X}=\one_{X}$ for every $t\in(0,\infty)$. Here, $P_{t}\one_{X}$ is defined by extending $P_{t}$ from $L^{2}(X,m)\cap L^{\infty}(X,m)$ to $L^{\infty}(X,m)$ \cite[p.~6]{CF12}.
\end{definition}

\begin{definition}\label{d:HKE}
Let $(X,d,m)$ be a metric measure space and $\Psi:[0,\infty)\to[0,\infty)$ be a continuous increasing bijection of $[0, \infty)$ onto itself. We say that a Dirichlet form $(\mathcal{E},\mathcal{F})$ on $L^{2}(X,m)$ satisfies the \emph{heat kernel estimates} \hypertarget{HKE} $\hyperlink{HKE}{\mathrm{HKE}(\Psi)}$, if there exists
		$C_{1},c_{2},c_{3}, \delta\in(0,\infty)$ and a heat kernel $\set{p_t}_{t\in(0,\infty)}$ of its semigroup $\{P_{t}\}_{t\in(0,\infty)}$ such that for any $t\in(0,\infty)$,
  		\begin{align}
		p_{t}(x,y) &\leq \frac{C_{1}}{m(B(x,\Psi^{-1}(t)))} \exp\left(-c_{2}t\Phi\left(c_{3}\frac{d(x, y)}{t}\right)\right)
		\qquad \mbox{for $m$-a.e.\ $x,y \in X$}\\
		\text{and }\ p_{t}(x,y) &\geq \frac{C_{1}^{-1}}{m(B(x,\Psi^{-1}(t)))}
		\qquad \mbox{for $m$-a.e.\ $x,y\in X$ with $d(x,y) \leq \delta \Psi^{-1}(t)$},
		\end{align}
		where  \begin{equation}\label{e.phi}
		\Phi(s):=\sup_{r\in(0,\infty)}\left({\frac{s}{r}-\frac{1}{\Psi(r)}}\right),\ s\in[0,\infty).
		\end{equation}
		For $\beta\in(1,\infty)$, we say that \hyperlink{HKE}{$\mathrm{HKE}(\beta)$} holds if $\hyperlink{HKE}{\mathrm{HKE}(\Psi)}$ holds with $\Psi(r)=r^{\beta}$, $r\in[0,\infty)$.
\end{definition}

Before we prove Theorem \ref{t.dic}, we need the following lemma on the relationship between $p$-modulus vanishing properties for different $p$.
\begin{lemma}\label{l.VMq>p}
	Let $(X,d,m)$ be a metric measure space. Let $1\leq p\leq q<\infty$. Then \begin{equation}
		 \hyperlink{VMp}{\mathrm{VM}_q}\text{ holds }\Longrightarrow  \hyperlink{VMp}{\mathrm{VM}_p}\text{ holds}.
	\end{equation}
\end{lemma}
\begin{proof}
Assume that $\hyperlink{VMp}{\mathrm{VM}_q}$ holds. Let $\Gamma$ be the family of all curves in $(X,d)$. Then $ \Mod_{q}(\Gamma)=0$. By \cite[Lemma 5.2.8]{HKST15}, there exists a nonnegative Borel function
$g:X\to [0,\infty]$ such that
\[
    \int_{X}g^{q}\dif m<\infty \text{ and }    \int_{\gamma}g\dif s=\infty,\ \text{for every $\gamma\in\Gamma$.}
\]
Define $h:=g^{q/p}$. Then $h$ is a nonnegative Borel function and $\int_X h^p\dif m= \int_X g^q\dif m  <\infty$. We claim that
\[
        \int_{\gamma}h\dif s=\infty,\ \text{for every $\gamma\in\Gamma$.}
\]
Let $\gamma$ be a locally rectifiable curve. We may assume that $\gamma$ has finite length, otherwise we can consider $\restr{\gamma}{[a,b]}\in\Gamma$ for suitable $a<b$. Then $ \int_{\gamma} g\one_{\{g\le 1\}}\dif s \leq \mathrm{Length}(\gamma)<\infty $. Hence, from $\int_{\gamma}g\dif s=\infty$, it follows that $\int_{\gamma} g\one_{\{g>1\}}\dif s=\infty$.
Since $q/p\ge 1$, we have $\restr{g^{q/p}}{\{g>1\}}\geq \restr{g}{\{g>1\}}$. Therefore $\int_{\gamma}h\dif s= \int_{\gamma}g^{q/p}\dif s\geq \int_{\gamma}g\,\one_{\{g>1\}}\dif s=\infty$. Applying \cite[Lemma 5.2.8]{HKST15} again, now with exponent $p$,
we obtain $ \Mod_{p}(\Gamma)=0$. In other words, $\hyperlink{VMp}{\mathrm{VM}_p}$ holds.
\end{proof}

Before proving Theorem~\ref{t.dic}, we recall the notion of a \emph{weak tangent} of a metric space used in \cite[Section~2.4]{KL04}. The Attouch--Wets convergence appearing below is recalled in Definition~\ref{d.AW} in Appendix~\ref{s.apdx}.

\begin{definition}
Let $(X,d)$ be a complete metric space. A complete pointed metric space $(Z,d_Z,z)$ is called a \emph{weak tangent} of $(X,d)$ if there exist sequences $(x_n)_{n\in\bN}\subset X$ and $(r_n)_{n\in\bN}\subset(0,\diam(X,d))$ such that
\begin{equation}
    (X,r_n^{-1}d,x_n)
    \xrightarrow{\mathrm{AW}}
    (Z,d_Z,z).
\end{equation}
\end{definition}

In Theorem~\ref{t.AW=GH} of Appendix~\ref{s.apdx}, we prove that Attouch--Wets convergence implies pointed Gromov--Hausdorff convergence, and that the converse holds after passing to a subsequence.

\begin{lemma}\label{l.dftgt}
Let $(X,d,m)$ be a metric measure space that is proper, length and Ahlfors $\alpha$-regular for some $\alpha\in[1,\infty)$. Let $(\sE,\sF)$ be a conservative, regular and strongly local Dirichlet form on $L^{2}(X,m)$ that satisfies \hyperlink{HKE}{$\mathrm{HKE}(\beta)$} with $\beta\in(1,\infty)$. Let $(Z,d_{Z})$ be a weak tangent of $(X,d)$. Then $(Z,d_Z)$ is a proper, length and Ahlfors $\alpha$-regular metric space, and there exists a conservative, regular and strongly local Dirichlet form $(\sE_{Z},\sF_{Z})$ on $L^{2}(Z,\sH_{\alpha})$ that satisfies \hyperlink{HKE}{$\mathrm{HKE}(\beta)$}, where $\sH_{\alpha}$ is the $\alpha$-Hausdorff measure on $(Z,d_{Z})$.
\end{lemma}
\begin{proof}
Let $C\in[1,\infty)$ be the constant such that \begin{equation}\label{e.ahlfc}
C^{-1}r^{\alpha}\leq m(B(x,r))\leq Cr^{\alpha},\ \text{for all }(x,r)\in X\times (0,\diam(X,d)).
\end{equation}
Suppose $(Z,d_{Z})$ is a weak tangent of $(X,d)$. By definition, there exist a point $z\in Z$, sequences $(x_n)_{n\in\bN}\subset X$ and $(r_n)_{n\in\bN}\subset (0,\diam(X,d))$ such that $(X,r_n^{-1}d,x_n)\xrightarrow{\mathrm{AW}} (Z,d_Z,z)$. Let $m_{n}=r_{n}^{-\alpha}\cdot m$. For any $(x,s)\in X\times (0,\infty)$, we denote by $B_n(x,s)$ the open ball in the metric space $(X,r_n^{-1}d)$. Therefore for any $(x,s)\in X\times (0,\diam(X,r_n^{-1}d))$, by \eqref{e.ahlfc},
\begin{equation}
m_{n}(B_{n}(x,s))=r_{n}^{-\alpha}\cdot m(B(x,r_{n}s))\in[C^{-1}s^{\alpha},Cs^{\alpha}].
\end{equation}
A similar calculation shows that $(X, r_n^{-1}d, m_{n}, r_n^{\beta-\alpha}\cdot \sE,\sF)$ satisfies \hyperlink{HKE}{$\mathrm{HKE}(\beta)$} with uniform constants independent of $n\in\bN$. By Theorem~\ref{t.AW=GH}-\ref{it.AW>GH}, $(X,r_n^{-1}d,x_n)\xrightarrow{\mathrm{AW}} (Z,d_Z,z)$ implies the {pointed Gromov--Hausdorff convergence} $(X,r_n^{-1}d,x_n)\xrightarrow{\mathrm{p-GH}} (Z,d_Z,z)$; see Definition \ref{d.pgh}. By \cite[Theorem 1.2]{Che26}, there exists a Radon measure $m_{Z}$ on $(Z,d_{Z})$ such that $(Z,d_{Z},m_{Z})$ is Ahlfors $\alpha$-regular, and there exists a conservative, regular and strongly local Dirichlet form $(\sE_{Z},\sF_{Z})$ on $L^{2}(Z,m_{Z})$ that satisfies \hyperlink{HKE}{$\mathrm{HKE}(\beta)$}. By \cite[Lemma 4.1.14 and Proposition 11.3.12]{HKST15}, $(Z,d_{Z})$ is proper and length.  By \cite[Exercise 8.11]{Hei01}, there exists a constant $C_{1}\in[1,\infty)$ such that \begin{equation}
C_{1}^{-1}\sH_{\alpha}(A)\leq m_{Z}(A)\leq C_{1}\sH_{\alpha}(A)\ \text{ for all Borel sets $A\subset Z$},
\end{equation}
which implies that $(\sE_{Z},\sF_{Z})$ is also a conservative, regular and strongly local Dirichlet form on $L^{2}(Z,\sH_{\alpha})$ that satisfies \hyperlink{HKE}{$\mathrm{HKE}(\beta)$} (with reference measure $\sH_{\alpha}$).
\end{proof}

\begin{proof}[Proof of Theorem \ref{t.dic}]
By Ahlfors regularity and {the heat kernel estimates} \hyperlink{HKE}{$\mathrm{HKE}(\beta)$}, we know by \cite[Lemma 4.7]{KM23} that $\beta\in[2,\infty)$. 

Given a metric measure space $(X,d,m)$, by \cite[Chapter 3]{FOT11}, every strongly local regular Dirichlet form $(\sE,\sF)$ on $L^{2}(X,m)$ corresponds to a regular local $2$-Dirichlet space $(X,d,m,\sE,\sF,\Gamma)$. By \cite[Proof of Theorem 1.2]{GHL15} or \cite[Proof of Theorem 3.2]{Lie15} (see also \cite[Remark 2.9]{KM20}), we know that the Poincar\'{e} inequality \hyperlink{PI}{$\mathrm{PI}_{2}(\beta)$} holds. By \cite[Proof of Theorem 1.2]{GHL15}, we know that \hyperlink{cap<}{$\mathrm{Cap}_{2}(\beta)_{\leq}$} holds. Therefore, Theorem \ref{t.main} implies that $\beta>2$ if and only if $\hyperlink{VMp}{\mathrm{VM}_{2}}$ holds. Now it suffices to show that $\dim_{\mathrm{ARC}}(X,d)=\alpha$ if and only if $\beta=2$. 

Suppose that $\dim_{\mathrm{ARC}}(X,d)=\alpha$. By \cite[Corollary 1.0.2]{KL04}, there exists a weak tangent of $(X,d)$, say $(Z,d_{Z})$, such that $\hyperlink{VMp}{\mathrm{VM}_{1}}$ fails for $(Z,d_{Z},\sH_{\alpha})$, where $\sH_{\alpha}$ is the $\alpha$-Hausdorff measure on $(Z,d_{Z})$. By Lemma \ref{l.dftgt}, there exists a strongly local regular Dirichlet form $(\sE_{Z},\sF_{Z})$ on $L^{2}(Z,\sH_{\alpha})$ that also satisfies \hyperlink{HKE}{$\mathrm{HKE}(\beta)$}. If $\beta>2$, then Theorem \ref{t.main}-\ref{it.main1} implies that $\hyperlink{VMp}{\mathrm{VM}_{2}}$ holds for  $(Z,d_{Z},\sH_{\alpha})$. Therefore, by Lemma \ref{l.VMq>p}, we know that $\hyperlink{VMp}{\mathrm{VM}_{1}}$ holds for $(Z,d_{Z},\sH_{\alpha})$. This leads to a contradiction and thus $\beta=2$.
 
Suppose that $\beta=2$. Let $(Z,d_{Z})$ be any weak tangent of $(X,d)$. By Lemma \ref{l.dftgt}, there exists a strongly local regular Dirichlet form $(\sE_{Z},\sF_{Z})$ on $L^{2}(Z,\sH_{\alpha})$ that also satisfies \hyperlink{HKE}{$\mathrm{HKE}(2)$}. By Theorem \ref{t.main}-\ref{it.main2}, we know that $\hyperlink{VMp}{\mathrm{VM}_{2}}$ fails on $(Z,d_{Z},\sH_{\alpha})$. That is, there is a weak tangent of $(X,d)$ with non-vanishing $2$-modulus. By \cite[Corollary 1.0.2]{KL04}, we know that $\dim_{\mathrm{ARC}}(X,d)=\alpha$. This completes the proof.
\end{proof}

As an application of Theorem~\ref{t.dic}, we obtain estimates for the Ahlfors-regular conformal dimension of generalized Sierpi\'{n}ski carpets. We begin by recalling their definition.

\begin{definition}[\text{Generalized Sierpi\'{n}ski carpet, {\cite[Definition 9.2]{KS26}}}] \label{d.GSC}
Let $D\in\bN\cap[2,\infty)$, $l\in\bN\cap[3,\infty)$. Let $S\subsetneq\{0,1,\ldots,l-1\}^{D}$ be non-empty.
\begin{enumerate}[label=\textup{({\arabic*})},align=right,leftmargin=*,topsep=5pt,parsep=0pt,itemsep=2pt]
	\item For each $j\in S$, define $f_{j}:\bR^{D}\to\bR^{D}$ \begin{equation}
	f_{j}(x):=l^{-1}x+l^{-1}j,\ x\in\bR^{D}.
\end{equation}
Let $Q_{0}:=[0,1]^{D}$ and $Q_{1}:=\bigcup_{j\in S}f_{j}(Q_{0})$ so that $Q_{1}\subsetneq Q_{0}$. Let $K$ be the unique self-similar set associated with $\{f_{j}\}_{j\in S}$. Let $F_{j}=\restr{f_{j}}{K}$, $j\in S$. 
	\item $K$ is called a \emph{generalized Sierpi\'{n}ski carpet} if and only if the following four conditions are satisfied:
\begin{enumerate}[label=\textup{(GSC{\arabic*})},align=right,leftmargin=*,topsep=5pt,parsep=0pt,itemsep=2pt] 
	\item \textup{(Symmetry)} $f(Q_1) = Q_1$ for any isometry $f$ of $\bR^{D}$ with $f(Q_0) = Q_0$.
	\item \textup{(Connectedness)} $Q_{1}$ is connected.
	\item \textup{(Non-diagonality)} The set $\mathrm{int}_{\bR^{D}}(Q_{1}\cap\prod_{k=1}^{D}[(i_{k}-\varepsilon_{k})l^{-1},(i_{k}+1)l^{-1}])$ is either empty or connected for any $(i_{k})_{k=1}^{D}\in\bZ^{D}$ and any $(\varepsilon_{k})_{k=1}^{D}\subset\{0,1\}^D$.
	\item \textup{(Borders included)} $[0,1]\times\{0\}^{D-1}\subset Q_{1}$.
\end{enumerate}
\end{enumerate}
\end{definition}

\begin{corollary}\label{c.GSC}
Let $(K,d)$ be a generalized Sierpi\'{n}ski carpet equipped with the Euclidean metric $d$. Then \begin{equation}
	\dim_{\mathrm{ARC}}(K,d)<\dim_{\mathrm{H}}(K,d).
\end{equation}
In other words, any generalized Sierpi\'{n}ski carpet is not minimal for its Ahlfors-regular conformal dimension.
\end{corollary}
Corollary \ref{c.GSC} follows directly from Theorem \ref{t.dic} and the known sub-Gaussian heat kernel estimates for generalized Sierpi\'{n}ski carpets \cite[Theorem 2.7 and the proof of Corollary 2.10]{Kaj23}. Alternatively, it can be obtained as a consequence of the following lemma together with \cite[Corollary 1.0.2]{KL04}. Our argument uses Theorem \ref{t.main}. For convenience, we use the metric $\ol{d}$ on $\bR^{D}$ defined by 
\begin{equation}\label{e.maxmt}
	\ol{d}((x_{1},\ldots,x_{D}),(y_{1},\ldots,y_{D})):=\max_{1\leq j\leq D}\abs{x_{j}-y_{j}},\ \text{for all }(x_{1},\ldots,x_{D}),(y_{1},\ldots,y_{D})\in\bR^{D}.
\end{equation}
The metric $\ol{d}$ is bi-Lipschitz equivalent to the standard Euclidean metric, and therefore is Ahlfors regular and $\dim_{\mathrm{ARC}}(K,\ol{d})=\dim_{\mathrm{ARC}}(K,d)$.
\begin{lemma}
	Let $(K,\ol{d})$ be a generalized Sierpi\'{n}ski carpet equipped with the metric $\ol{d}$ defined by \eqref{e.maxmt}. Suppose $\dim_{\mathrm{H}}(K,\ol{d})=d_{\mathrm{f}}$. Let $(Z,d_{Z})$ be any weak tangent of $(K,\ol{d})$. Then $(Z,d_{Z})$ is Ahlfors $d_{\mathrm{f}}$-regular, and \ref{VMp} holds for $(Z,d_{Z},\sH_{d_{\mathrm{f}}})$ for every $p\in[1,\infty)$, where $\sH_{d_{\mathrm{f}}}$ is the $d_{\mathrm{f}}$-Hausdorff measure on $(Z,d_{Z})$.
\end{lemma}

\begin{proof}
Let $p\in(\dim_{\mathrm{ARC}}(K,d),\infty)$. By \cite[Corollary 9.4]{KS26}, there exists a $p$-energy form $(\sE_{p},\sF_{p})$ with $p$-energy measures $\{\carre{f}\}_{f\in\sF_{p}}$ such that $(K, \ol{d}, m, \sE_{p},\sF_{p},\Gamma_{p})$ is a regular local $p$-Dirichlet space. Let $d_{\mathrm{w},p}$ be the $p$-walk dimension of $(K,\ol{d})$ (see \cite[Definition 8.18]{KS26}). By \cite[Eq.~(8.11)]{KS26} and \cite[Proposition 3.3]{Kig23}, we have $d_{\mathrm{w},p}>d_{\mathrm{f}}$ and therefore the metric measure space $(K,\ol{d},m)$ satisfies the \emph{slow volume regular condition} $\mathrm{SVR}(d_{\mathrm{f}},d_{\mathrm{w},p})$ defined in \cite{Yan25a}. By \cite[Theorem 8.19]{KS26}, it also satisfies a \emph{resistance estimate} $\mathrm{R}(d_{\mathrm{f}},d_{\mathrm{w},p})$ defined in \cite{Yan25a}. Therefore, by \cite[Theorem 2.3]{Yan25a}, we know that \hyperlink{PI}{$\mathrm{PI}(d_{\mathrm{w},p})$} holds. By \cite[Theorem 9.8]{KS26}, we know that $d_{\mathrm{w},p}>p$, and therefore by Theorem \ref{t.main}-\ref{it.main1}, we know that \ref{VMp} holds on $(K,\ol{d},m)$.
	
	The remainder of the proof could probably also be obtained by adapting the method of Lupo--Krebs and Pajot \cite[Théorème 5.4]{LKP04}. For completeness, we give a full proof below. Let $(Z,d_{Z},z)$, where $z\in Z$, be a weak tangent of this generalized Sierpi\'{n}ski carpet $(K,\ol{d})$, that is, there exist sequences $(x_{n})_{n\in\bN}\subset K$ and $(r_{n})_{n\in\bN}\subset(0,\diam(K,\ol{d}))$ such that $(K,r_{n}^{-1}\ol{d}, x_{n})\xrightarrow{\mathrm{AW}}(Z,d_{Z},z)$. 
	By Theorem \ref{t.AW=GH}-\ref{it.AW>GH}, we know that \begin{equation}\label{e.GSC1}
		(K,r_{n}^{-1}\ol{d}, x_{n})\xrightarrow{\mathrm{p-GH}}(Z,d_{Z},z).
	\end{equation}
	By \cite[Lemma 2.4.4]{KL04}, the metric space $(Z,d_{Z})$ is also Ahlfors $d_{\mathrm{f}}$-regular. Since the metric space $(K,r_{n}^{-1}\ol{d})$ is isometric to $(r_{n}^{-1}(K-x_{n}),\ol{d})$, we can rephrase \eqref{e.GSC1} as \begin{equation}\label{e.GSC1.1}
		(r_{n}^{-1}(K-x_{n}),\ol{d},0)\xrightarrow{\mathrm{p-GH}}(Z,d_{Z},z),
	\end{equation}
	where $0\in K-x_{n}\in \bR^{D}$ is the origin of $\bR^{D}$. Let $s\in(0,\infty)$. By \cite[Proposition 11.3.12]{HKST15}, we know that \eqref{e.GSC1.1} implies the Gromov--Hausdorff convergence of compact metric spaces:\begin{equation}
		(r_{n}^{-1}(K-x_{n})\cap \ol{B_{\ol{d}}(0,s)},\ol{d})\xrightarrow{\mathrm{GH}}(\ol{B_{d_{Z}}(z,s)},d_{Z});
	\end{equation}
	see \cite[p.~309]{HKST15} for the definition of Gromov--Hausdorff convergence.
	Denote \begin{equation}
		K_{n}^{(s)}:=r_{n}^{-1}(K-x_{n})\cap \ol{B_{\ol{d}}(0,s)}\subset\bR^{D},\ n\in\bN.
	\end{equation} By passing to a subsequence, we may assume that \begin{equation}\label{e.GSC1.2}
		d_{\mathrm{GH}}(K_{n}^{(s)},K_{n+1}^{(s)})\leq 2^{-n}\ \text{and} \ (K_{n}^{(s)},\ol{d})\xrightarrow{\mathrm{GH}}(\ol{B_{d_{Z}}(z,s)},d_{Z}),
	\end{equation} 
	where $d_{\mathrm{GH}}$ is the Gromov--Hausdorff distance defined in \cite[Eq.~(11.1.6)]{HKST15}. Consider the distance between two compact subsets $X$ and $Y$ of $\bR^{D}$ defined by \begin{equation}
		d_{\mathrm{iso}}(X,Y):=\inf\Sett{d_{\mathrm{H}}(X,T(Y))}{T:\bR^{D}\to\bR^{D} \text{ is an isometry}}.
	\end{equation}
	By \cite[Theorem~2]{Mem08}, we know that there is a constant $C\in(1,\infty)$ independent of $n$ and $s$, such that $d_{\mathrm{iso}}(K_{n}^{(s)},K_{n+1}^{(s)})\leq Cs2^{-n}$ for all $n\in\bN$ and all $s\in(0,\infty)$. By definition of $d_{\mathrm{iso}}$, for each $n\in\bN$, there exists an isometry $\wt{T}_{n+1}:\bR^{D}\to\bR^{D}$ such that \begin{equation}\label{e.GSC2}
		d_{\mathrm{H}}(K_{n}^{(s)},\wt{T}_{n+1}K_{n+1}^{(s)})\leq Cs2^{-n+1}.
	\end{equation}
	Let $\wt{T}_{1}:\bR^{D}\to\bR^{D}$ be the identity map. Define $T_{n}:=\wt{T}_{1}\circ \cdots\circ \wt{T}_{n},\ n\in\bN$. Then each $T_{n}$ is an isometry on $\bR^{D}$. By \eqref{e.GSC2}, we have \begin{equation}
		d_{\mathrm{H}}(T_{n}K_{n}^{(s)},T_{n+1}K_{n+1}^{(s)})=d_{\mathrm{H}}(K_{n}^{(s)},\wt{T}_{n+1}K_{n+1}^{(s)})\leq Cs2^{-n+1}.
	\end{equation}
	Therefore $\{T_{n}K_{n}^{(s)}\}_{n\in\bN}$ is a Cauchy sequence under the Hausdorff metric $d_{\mathrm{H}}$. By the completeness of $d_{\mathrm{H}}$, there exists a compact subset $\wt{W}_{s}\subset \bR^{D}$ such that $T_{n}K_{n}^{(s)}\to \wt{W}_{s}$ under the Hausdorff metric. By \eqref{e.GSC1.2} and the uniqueness of Gromov--Hausdorff limit \cite[Theorem 11.1.14]{HKST15}, we know that $(\ol{B_{d_{Z}}(z,s)},d_{Z})$ is isometric to $(\wt{W}_{s},\ol{d})$.

	For each $n\in\bN$, define $y_{n}:=T_{n}(0)$. Since $y_{n}$ has a uniform bounded distance to $\wt{W}_{s}$, there exists a $y$ such that $\lim_{n\to\infty}y_{n}=y$. Define a function $\psi_{n}:\bR^{D}\to\bR^{D}$ by 
	\begin{equation}
		\psi_{n}(x):=r_{n}^{-1}\cdot T_{n}(x-x_n)-y,\ x\in\bR^{D}.
	\end{equation}
	Then $\psi_{n}$ is a similarity map on $\bR^{D}$ and $T_{n}K_{n}^{(s)}$ is isometric to $\psi_{n}(K)\cap \ol{B(0,s)}$. Since $\lim_{n\to\infty}y_{n}=y$, we know the Hausdorff limit of $\psi_{n}(K)\cap \ol{B(0,s)}$ is $\wt{W}_{s}-y$. By the scaling property of the Hausdorff metric, we know that \begin{equation}
		(2s)^{-1}\psi_{n}(K)\cap \ol{B(0,2^{-1})} \text{ converges to }W_{s}:=(2s)^{-1}(\wt{W}_{s}-y)\ \text{in the Hausdorff metric}.
	\end{equation}
	By translating $\ol{B(0,2^{-1})}$ to $[0,1]^{D}\subset\bR^{D}$, we know that $(2s)^{-1}\psi_{n}(K)\cap \ol{B(0,2^{-1})}$ form a sequence of \emph{minisets} of the generalized Sierpi\'{n}ski carpet $K$, according to the definition in \cite[Definition 1.1-(1)]{Day22}. Therefore $W_{s}$ is a \emph{microset} of $K$ (see \cite[Definition 1.1-(2)]{Day22}). Since $K$ satisfies the \emph{open set condition}, by \cite[Theorem 1.3]{Day22}, there exists $N\in\bN$ that depends only on the affine maps $\{F_{j}\}_{j\in S}$ in Definition \ref{d.GSC}, and there exist expanding similarity maps $(\varphi_{j})_{j=1}^{N}$ on $\bR^{D}$, such that $W_{s}\subset \bigcup_{j=1}^{N}\varphi_{j}(K)$.

	Since $(Z,d_{Z})$ satisfies the chain condition and is Ahlfors $d_{\mathrm{f}}$-regular, we know that $(W_{s},\ol{d})$ is also Ahlfors $d_{\mathrm{f}}$-regular. In order to prove \ref{VMp} for $(W_{s},\ol{d})$, it suffices to prove \ref{VMp} for the metric space $(\bigcup_{j=1}^{N}\varphi_{j}(K),\ol{d})$. Let $\sH_{d_\mathrm{f}}^{\bR^{D}}$ be the $d_{\mathrm{f}}$-Hausdorff measure on $(\bR^{D},\ol{d})$. Since $(K,\ol{d})$ satisfies \ref{VMp} and $\varphi_{j}$ is a similarity map, we know that $(\varphi_{j}(K),\ol{d})$ satisfies \ref{VMp} for all $j\in\{1,\ldots,N\}$. By \cite[Lemma 5.2.8]{HKST15}, there is a function $\rho_{j}\in L^{p}(\varphi_{j}(K),\sH_{d_\mathrm{f}}^{\bR^{D}})$ such that $\int_{\gamma}\rho_j\dif s=\infty$ for every locally rectifiable curve $\gamma$ in $\varphi_{j}(K)$.  We extend $\rho_{j}$ to be a function on $\bR^{D}$ by setting $\rho_{j}(x)=0$ when $x\in\bR^{D}\setminus \varphi_{j}(K)$. Define $\rho:=\sum_{j=1}^{N}\rho_{j}$. Clearly, $\rho\in L^{p}(W_{s},\sH_{d_\mathrm{f}}^{\bR^{D}})$. Let $\gamma:[0,1]\to W_{s}$ be a locally rectifiable curve in $(W_{s},\ol{d})$. Then $\{\gamma^{-1}(\varphi_{j}(K))\}_{j=1}^{N}$ are finitely many closed subsets, whose union is $[0,1]$. By passing to an arc-length parametrization \cite[p.~124]{HKST15}, we may assume that $\gamma$ has no constant restriction to a nondegenerate subinterval. By the Baire category theorem \cite[Theorem 2.1]{Bre11}, there exists $j_{0}\in\{1,\ldots,N\}$ such that $\gamma^{-1}(\varphi_{j_0}(K))$ contains an interval $[a,b]$ with $a<b$. Therefore \begin{equation}
		\int_{\gamma}\rho\dif s\geq\int_{{\gamma}|_{[a,b]}}\rho_{j_{0}}\dif s=\infty.
	\end{equation} 
	By \cite[Lemma 5.2.8]{HKST15} again, we know that $(\bigcup_{j=1}^{N}\varphi_{j}(K),\ol{d})$ has vanishing $p$-modulus, and hence so do $(W_{s},\ol{d})$ and $(\ol{B_{d_{Z}}(z,s)},d_{Z})$ as they are isometric. Since this holds for all $s\in(0,\infty)$, a similar argument shows that $(Z,d_{Z})$ satisfies \ref{VMp}, for every $p\in(\dim_{\mathrm{ARC}}(K,\ol{d}),\infty)$. By Lemma \ref{l.VMq>p}, we know that $(Z,d_{Z})$ satisfies \ref{VMp} for every $p\in[1,\infty)$.
	\end{proof}

\begin{appendices}
\section{Equivalence of two versions of convergence of metric space}\label{s.apdx}

In this appendix, we prove the equivalence, up to passing to a subsequence, between two notions of convergence for pointed metric spaces. The notion used in \cite[Section~2.4]{KL04} differs from the standard notion of pointed Gromov--Hausdorff convergence. Their equivalence up to subsequences is asserted in \cite[the first paragraph of Section~2.2]{Kei03}, but we have not been able to locate a proof. We therefore give a detailed argument. We begin by recalling the notion of convergence for subsets of a fixed ambient space used in \cite[Section~2.4]{KL04}.

\begin{definition}
Let $(M,d_M)$ be a metric space, and let $F_i,F\subset M$ be nonempty closed subsets.  We say that $F_i$ converges \emph{locally in Hausdorff sense to} $F$ if, for every $q\in M$ and every $R>0$,
\begin{equation}\label{eq:lh1}
        \lim_{i\to\infty}\sup_{z\in F_i\cap B_M(q,R)}\dist_M(z,F)=0,
\end{equation}
 and
\begin{equation}\label{eq:lh2}
        \lim_{i\to\infty}\sup_{z\in F\cap B_M(q,R)}\dist_M(z,F_i)=0.
\end{equation}
If the set over which the supremum is taken is empty, the supremum is interpreted as $0$.
\end{definition}

The notion of convergence used in \cite[Definition 2.4.2]{KL04} to define weak tangents is a variant of Attouch--Wets convergence \cite{ALW91}. For the purposes of this appendix, we shall still call it Attouch--Wets convergence.

\begin{definition}[Attouch--Wets convergence]\label{d.AW}
A sequence of pointed metric spaces $\{(X_{j},d_{j},p_{j}):j\in\bN\}$ is said to \emph{Attouch--Wets converge} to a pointed metric space $(X_{\infty},d_{\infty},p_{\infty})$, if there exist a pointed metric space $(M,d_M,p)$ and isometric embeddings
\[
        \iota_j:X_{j}\to M,
        \qquad
        \iota_{\infty}:X_{\infty}\to M,
\]
such that
\[
        \iota_j(p_j)=p,
        \qquad
        \iota_{\infty}(p_{\infty})=p,
\]
and the closed sets $\iota_j(X_j)$ converge locally in Hausdorff sense to $\iota_{\infty}(X_{\infty})$ inside $(M,d_{M})$. 

In this case, we write $(X_{j},d_{j},p_{j})\xrightarrow{\mathrm{AW}}(X_{\infty},d_{\infty},p_{\infty})$.
\end{definition}

The following definition of pointed Gromov--Hausdorff convergence is taken from \cite[Definition 11.3.1]{HKST15}.

\begin{definition}[Pointed Gromov--Hausdorff convergence]\label{d.pgh}
A sequence of pointed separable metric spaces $\{(X_{j},d_{j},p_{j}):j\in\bN\}$ is said to \emph{pointed Gromov--Hausdorff converge} to a pointed separable metric space $(X_{\infty},d_{\infty},p_{\infty})$ if, for each $r>0$ and $\epsilon\in (0,r)$, there exists a $j_{0}\in\bN$ such that for each $j\geq j_{0}$ there is a map $f_{j}^{\epsilon}: B_{j}(p_{j},r)\rightarrow X_{\infty}$ satisfying:
\begin{enumerate}[label=\textup{(\arabic*)},align=right,leftmargin=*,topsep=5pt,parsep=0pt,itemsep=2pt]
    \item $f_{j}^{\epsilon}(p_{i})=p_{\infty}$;
    \item $|d(f_{j}^{\epsilon}(x),f_{j}^{\epsilon}(y))-d_{j}(x,y)|<\epsilon$ for all $x,y\in B_{j}(p_{j},r)$;
    \item $B_{\infty}(p_{\infty},r-\epsilon)\subset  N_{\epsilon}(f_{j}^{\epsilon}(B_{j}(p_{j},r)))$, where $N_{\epsilon}(A)$ is the $\epsilon$-neighborhood of $A$.
\end{enumerate}
In this case, we write $(X_{j},d_{j},p_{j})\xrightarrow{\mathrm{p-GH}}(X_{\infty},d_{\infty},p_{\infty})$.
\end{definition}

The following theorem is the main result in this appendix.
\begin{theorem}\label{t.AW=GH}
Let $\{(X_{j},d_{j},p_{j}):j\in\bN\}$ and $(X_{\infty},d_{\infty},p_{\infty})$ be complete pointed separable metric spaces.  Then 
\begin{enumerate}[label=\textup{({\arabic*})},align=right,leftmargin=*,topsep=5pt,parsep=0pt,itemsep=2pt]
	\item\label{it.AW>GH} if $(X_{n},d_{n},p_{n})\xrightarrow{\mathrm{AW}}(X_{\infty},d_{\infty},p_{\infty})$, then $(X_{n},d_{n},p_{n})\xrightarrow{\mathrm{p-GH}}(X_{\infty},d_{\infty},p_{\infty})$;
	\item\label{it.GH>AW} if $(X_{n},d_{n},p_{n})\xrightarrow{\mathrm{p-GH}}(X_{\infty},d_{\infty},p_{\infty})$, then there is a subsequence $\{j_{k}\}_{k\in\bN}$ such that $j_{k}\uparrow\infty$ and $(X_{j_{k}},d_{j_{k}},p_{j_{k}})\xrightarrow{\mathrm{AW}}(X_{\infty},d_{\infty},p_{\infty})$.
\end{enumerate}
\end{theorem}

We first prove that the Attouch--Wets convergence implies pointed Gromov--Hausdorff convergence.
\begin{proof}[Proof of Theorem \ref{t.AW=GH}-\ref{it.AW>GH}]
	By definition, there exist a pointed metric space $(M,d_M,p)$ and pointed isometric embeddings $ \iota_i:X_i\to M$ and $\iota_{\infty}:X_{\infty}\to M$ with $\iota_i(p_i)=p$ and $ \iota_{\infty}(p_{\infty})=p$,
such that
\[
        F_i:=\iota_i(X_i)
        \ \text{converges locally in Hausdorff sense to}\ 
        F_{\infty}:=\iota_{\infty}(X_{\infty}) \text{ in } (M,d_{M}).
\]
Fix $r>0$ and $0<\varepsilon<r$.  Choose a number $\delta>0$ such that $2\delta<\varepsilon$. By local Hausdorff convergence, there exists $i_0\in\bN$ such that for every $i\geq i_0$,
\begin{equation}\label{e.Fdelta}
        \sup_{z\in F_i\cap B_M(p,r)}\dist_M(z,F_{\infty})<\delta,\ \text{ and }\  \sup_{z\in F\cap B_M(p,r)}\dist_M(z,F_i)<\delta.
\end{equation}
We now define, for each such $j$, a map $f_j:B_{j}(p_j,r)\to X_\infty$. Let $ f_j(p_j):=p_\infty$. For any $x\in B_{j}(p_j,r)\setminus\{p_{j}\}$, since $\iota_{j}$ is an isometry, \begin{equation}
	 d_M(p,\iota_j(x))=d_M(\iota_j(p_j),\iota_j(x))=d_j(p_j,x)<r.
\end{equation}
we know that $\iota_{j}(x)\in F_j\cap B_M(p,r)$. By \eqref{e.Fdelta}, we may choose $f_j(x)\in X_\infty$ such that
\begin{equation}\label{eq:choice-fj}
 d_M\bigl(\iota_j(x),\iota_\infty(f_j(x))\bigr)<\delta .
\end{equation}
For $x=p_j$, the same estimate holds evidently. Let $x,y\in B_j(p_j,r)$. Since the embeddings are isometric,
\begin{align}
&\phantom{\ \leq}\bigl|d_\infty(f_j(x),f_j(y))-d_j(x,y)\bigr| \\
&=\bigl|d_M(\iota_\infty(f_j(x)),\iota_\infty(f_j(y)))-d_M(\iota_j(x),\iota_j(y))\bigr| \\
&\le d_M(\iota_\infty(f_j(x)),\iota_j(x))+d_M(\iota_\infty(f_j(y)),\iota_j (y))<2\delta<\varepsilon .
\end{align}
It remains to prove almost surjectivity.  Let $z\in B_\infty(p_\infty,r-\varepsilon)$.  Then $\iota_\infty(z)\in F_\infty\cap B_M(p,r)$.  By \eqref{e.Fdelta}, choose $y\in X_j$ such that $d_M(\iota_\infty (z),\iota_j(y))<\delta$.  This point actually lies in $B_j(p_j,r)$, because
\[
 d_j(p_j,y)=d_M(p,\iota_j(y))\le d_M(p,\iota_\infty (z))+d_M(\iota_\infty(z),\iota_j(y))<(r-\varepsilon)+\delta<r.
\]
Therefore $f_j(y)$ is defined, and by \eqref{eq:choice-fj},
\[
 d_\infty(z,f_j(y))=d_M(\iota_\infty (z),\iota_\infty (f_j(y)))\le d_M(\iota_\infty (z),\iota_j(y))+d_M(\iota_j(y),\iota_\infty(f_j(y)))<2\delta<\epsilon .
\]
Since $z$ was arbitrary, we have $B_{\infty}(p_{\infty},r-\epsilon)\subset  N_{\epsilon}(f_{j}(B_{j}(p_{j},r)))$. 
\end{proof}

We then prove the reverse direction. For each integer $k\in\bN$, apply the map definition with $r=k$ and $\eta_k:=k^{-1}$.  Choose indices $j_k$ strictly increasing and maps $f_k:B_{j_k}(p_{j_k},k)\to X_\infty$ so that \begin{equation}\label{e.fkdis0}
	f_k(p_{j_k})=p_\infty,
\end{equation}
\begin{equation}\label{e.fkdis}
\bigl|d_\infty(f_k(u),f_k(v))-d_{j_k}(u,v)\bigr|<\eta_k,
\quad\text{for all }u,v\in B_{j_k}(p_{j_k},k),
\end{equation}
and
\begin{equation}\label{e.fkdis1}
B_\infty(p_\infty,k-\eta_k)
\subset N_{\eta_k}\bigl(f_k(B_{j_k}(p_{j_k},k))\bigr).
\end{equation}

Our strategy is to `glue' $X_{j_{k}}$ in a suitable way; see \cite[Chapter 3, Section 3.1]{BBI01}. For each $k\in\bN$, we define a function $\lambda_{k}: B_{j_k}(p_{j_k},k)\to[0,\infty)$ by \begin{equation}
	\lambda_{k}(z):=\eta_{k}\one_{B_{j_k}(p_{j_k},k)\setminus\{p_{j_{k}}\}}(z):=\begin{cases}
		0\quad &\text{ if }z=p_{j_{k}},\\
		\eta_{k}\quad &\text{ if }z\in B_{j_k}(p_{j_k},k)\setminus\{p_{j_{k}}\},
	\end{cases}
\end{equation}
and define a function $\rho_{k}:X_{\infty}\times X_{j_{k}}\to[0,\infty)$ by \begin{equation}\label{e.rhok}
\rho_k(x,y):=\inf_{u\in B_{j_k}(p_{j_k},k)}\bigl(d_\infty(x,f_k(u))+\lambda_k(u)+d_{j_k}(u,y)\bigr),\ (x,y)\in X_{\infty}\times X_{j_{k}}.
\end{equation}
\begin{lemma}\label{l.puns}
	For every $k\in\bN$, the functions $\lambda_{k}$ and $\rho_k$ have the following properties.\begin{enumerate}[label=\textup{(\roman*)},align=right,leftmargin=*,topsep=5pt,parsep=0pt,itemsep=2pt]
	\item\label{it.pun1} $\lambda_{k}$ bounds the distortion: \begin{equation}\label{e.lamdis}
		\bigl|d_\infty(f_k(u),f_k(v))-d_{j_k}(u,v)\bigr|\le \lambda_k(u)+\lambda_k(v),\ \text{ for all $u,v\in B_{j_k}(p_{j_k},k)$}.
	\end{equation}
	\item\label{it.pun2} $\rho_k(p_\infty,p_{j_k})=0$, $\rho_k(x,p_{j_k})=d_\infty(x,p_\infty)$ for every $x\in X_\infty$, and $\rho_k(p_\infty,y)=d_{j_k}(p_{j_k},y)$ for every $y\in X_{j_k}$.
	\item\label{it.pun3} $\rho_k$ is $1$-Lipschitz in each variable:\begin{equation}\label{e.pun31}
		|\rho_k(x,y)-\rho_k(x',y)|\le d_\infty(x,x'),\ \text{ for all }x,x{'}\in X_{\infty},\ y\in X_{j_{k}},
	\end{equation}  \begin{equation}\label{e.pun32}
		|\rho_k(x,y)-\rho_k(x,y')|\le d_{j_k}(y,y'),\ \text{ for all }y,y{'}\in X_{j_{k}},\ x\in X_{\infty}.
	\end{equation} 
\item\label{it.pun4} For all $x,x'\in X_\infty$ and $y,y'\in X_{j_k}$,
\begin{equation}\label{e.adm1}
d_\infty(x,x')\le \rho_k(x,y)+d_{j_k}(y,y')+\rho_k(x',y'),
\end{equation}
\begin{equation}\label{e.adm2}
d_{j_k}(y,y')\le \rho_k(x,y)+d_\infty(x,x')+\rho_k(x',y').
\end{equation}
\item\label{it.pun5} If $y\in B_{j_k}(p_{j_k},k)$, then $\rho_k(f_k(y),y)\le\lambda_k(y)\le\eta_k$.
\end{enumerate}
\end{lemma}
\begin{proof}
	\begin{enumerate}[label=\textup{(\roman*)},align=right,leftmargin=*,topsep=5pt,parsep=0pt,itemsep=2pt]
	\item[\ref{it.pun1}] If $u=v=p_{j_k}$ the statement is trivial; if exactly one of $u,v$ is $p_{j_k}$, then the right-hand side is $\eta_k$; if neither is $p_{j_k}$, then the right hand side is $2\eta_k$, which is even larger by \eqref{e.fkdis}.
	\item[\ref{it.pun2}] 
	If $y=p_{j_k}$, then choose $u=p_{j_{k}}$ in \eqref{e.rhok} and we have $\rho_k(x,p_{j_k})\leq d_\infty(x,p_\infty)$. On the other hand, for every $u\in B_{j_k}(p_{j_k},k)$, \begin{align}
		 d_\infty(x,p_\infty)&\le d_\infty(x,f_k(u))+d_\infty(f_k(u),p_\infty)\text{ (triangle inequality)}\\
 &\overset{\eqref{e.lamdis}}{\le} d_\infty(x,f_k(u))+\lambda_k(u)+d_{j_k}(u,p_{j_k}) \text{ (choose $v=p_{j_{k}}$ in \eqref{e.lamdis})}
	\end{align}
	Taking infimum over $u\in B_{j_k}(p_{j_k},k)$, we see that $ d_\infty(x,p_\infty)\leq \rho_k(x,p_{j_k})$. 
	
	The identity $\rho_k(p_\infty,y)=d_{j_k}(p_{j_k},y)$ is proved in the same way. In particular, if $(x,y)=(p_\infty,p_{j_k})$, then $\rho_k(p_\infty,p_{j_k})=0$.
	\item[\ref{it.pun3}] For each $u\in B_{j_k}(p_{j_k},k)$, we have by triangle inequality that \begin{equation}
		 d_\infty(x,f_k(u))+\lambda_k(u)+d_{j_k}(u,y)
 \le d_\infty(x,x')+d_\infty(x',f_k(u))+\lambda_k(u)+d_{j_k}(u,y).
	\end{equation}
	Taking infimum over $u\in B_{j_k}(p_{j_k},k)$, we see that \begin{equation}
		\rho_k(x,y)-\rho_k(x',y)\le d_\infty(x,x');
	\end{equation}
	interchanging $x$ and $x'$ gives the reverse inequality. The proof in the $y$-variable is identical.
	\item[\ref{it.pun4}] First fix $u,v\in B_{j_k}(p_{j_k},k)$. Then
	\begin{align}
 &\phantom{\ \leq}d_\infty(x,x')\\
 &\le d_\infty(x,f_k(u))+d_\infty(f_k(u),f_k(v))+d_\infty(f_k(v),x') \\
 &\overset{\eqref{e.lamdis}}{\le} d_\infty(x,f_k(u))+d_{j_k}(u,v)+\lambda_k(u)+\lambda_k(v)+d_\infty(f_k(v),x') \\
 &\le d_\infty(x,f_k(u))+\lambda_k(u)+d_{j_k}(u,y)+d_{j_k}(y,y')+d_{j_k}(y',v)+\lambda_k(v)+d_\infty(f_k(v),x');
\end{align}
the first and third inequalities use triangle inequality. Taking the infimum first over $u$ and then over $v$ gives \eqref{e.adm1}.  The proof of \eqref{e.adm2} is symmetric. For $u,v\in B_{j_k}(p_{j_k},k)$,
\begin{align}
&\phantom{\ \leq}d_{j_k}(y,y')\\
 &\leq d_{j_k}(y,u)+d_{j_k}(u,v)+d_{j_k}(v,y') \\
&\overset{\eqref{e.lamdis}}{\le} d_{j_k}(y,u)+d_\infty(f_k(u),f_k(v))+\lambda_k(u)+\lambda_k(v)+d_{j_k}(v,y') \\
 &\leq d_{j_k}(y,u)+\lambda_k(u)+d_\infty(f_k(u),x)+d_\infty(x,x')+d_\infty(x',f_k(v))+\lambda_k(v)+d_{j_k}(v,y').
\end{align}
Taking infimum over $u$ and $v$ gives \eqref{e.adm2}. 
\item[\ref{it.pun5}] If $y\in B_{j_k}(p_{j_k},k)$, choosing $x=f_k(y)$ and $u=y$ in \eqref{e.rhok} gives $\rho_k(f_k(y),y)\le\lambda_k(y)$.
\qedhere
\end{enumerate}
\end{proof}

\begin{lemma}\label{l.Mdp}
 Define a set \begin{equation}
	{M}_{0}:=X_{\infty}\sqcup\bigsqcup_{k\in\bN}X_{j_{k}}.
\end{equation}
Define an equivalence relationship $\sim$ on $M_{0}$ by \begin{equation}
\begin{cases}
	x\sim x,\ \text{ for any $x\in {M}_{0}\setminus\{p_{\infty},p_{j_1},\ldots\}$},\\
	x\sim y,\ \text{ for every $x,y\in \{p_{\infty},p_{j_1},\ldots\}$}.
\end{cases}
\end{equation}
Let $M:={M}_{0}/\sim$ be the quotient space by identifying all basepoints $\{p_{\infty},p_{j_1},\ldots\}$ as a common point. Denote $p=[p_{j_{1}}]=\ldots=[p_{\infty}]\in M$. Define a function $d_{M}: M\times M\to[0,\infty)$ as follows.\begin{equation}\label{e.defdm}
	d_{M}([x],[y]):=\begin{dcases}
		\qquad d_{\infty}(x,y)\quad &\text{ if }(x,y)\in X_{\infty}\times X_{\infty},\\
		\qquad d_{j_{k}}(x,y)\quad &\text{ if }(x,y)\in X_{j_{k}}\times X_{j_{k}} \text{ for some $k\in\bN$},\\
		\qquad \rho_{k}(x,y)\quad &\text{ if }(x,y)\in X_{\infty}\times X_{j_{k}} \text{ for some $k\in\bN$},\\
		\qquad \rho_{k}(y,x)\quad &\text{ if }(x,y)\in  X_{j_{k}}\times X_{\infty} \text{ for some $k\in\bN$},\\
		\inf_{z\in X_\infty}\bigl(\rho_k(z,x)+\rho_{l}(z,y)\bigr) &\text{ if }(x,y)\in  X_{j_{k}}\times X_{j_{l}} \text{ for some $k,l\in\bN$}.
	\end{dcases}
\end{equation}

	Then the function $d_{M}: M\times M\to[0,\infty)$ is well-defined and is a metric on $M$.
\end{lemma}
\begin{proof}
	By Lemma \ref{l.puns}-\ref{it.pun2}, the function $d_{M}$ is independent of the choices of the representatives, and hence is well-defined.
	
	Symmetry and finiteness of $d_{M}$ are immediate from the definition. We prove the triangle inequality and then separation of points.
	
	Fix three points $[x]$, $[y]$ and $[z]$ in $M$.
	\begin{enumerate}[label=\textit{Case {\arabic*}.},align=right,leftmargin=*,topsep=5pt,parsep=0pt,itemsep=2pt]
	\item \textit{If $x,y,z\in X_{\infty}$ or if $x,y,z\in X_{j_{k}}$ for some $k\in\bN$.}
	
	 In this case, the triangle inequality follows from the triangle inequality on $(X_{\infty},d_{\infty})$ or on $(X_{j_{k}},d_{j_{k}})$.
	\item \textit{If $x,y\in X_\infty$ and $z\in X_{j_k}$ for some $k\in\bN$.}
	
	In this case, \begin{equation}
		d_{M}([x],[y])\overset{\eqref{e.defdm}}{=}d_{\infty}(x,y)\overset{\eqref{e.adm1}}{\leq}\rho_{k}(x,z)+\rho_{k}(y,z)\overset{\eqref{e.defdm}}{=}d_{M}([x],[z])+d_{M}([z],[y]).
	\end{equation}
	\begin{equation}
		d_{M}([x],[z])\overset{\eqref{e.defdm}}{=}\rho_{k}(x,z)\overset{\eqref{e.pun31}}{\leq} d_{\infty}(x,y)+\rho_{k}(y,z)\overset{\eqref{e.defdm}}{=}d_{M}([x],[y])+d_{M}([z],[y]).
	\end{equation}
	and similarly we have $d_{M}([y],[z])\leq d_{M}([y],[x])+d_{M}([x],[z])$.
	\item  \textit{If $x,y\in X_{j_k}$ for some $k\in\bN$ and $z\in X_\infty$.}
	
	In this case, \begin{equation}
		d_{M}([x],[y])\overset{\eqref{e.defdm}}{=}d_{j_{k}}(x,y)\overset{\eqref{e.adm2}}{\leq}\rho_{k}(z,x)+\rho_{k}(z,y)\overset{\eqref{e.defdm}}{=}d_{M}([x],[z])+d_{M}([z],[y]).
	\end{equation}
	\begin{equation}
		d_{M}([x],[z])\overset{\eqref{e.defdm}}{=}\rho_{k}(z,x)\overset{\eqref{e.pun32}}{\leq}\rho_{k}(z,y)+d_{j_{k}}(x,y)\overset{\eqref{e.defdm}}{=}d_{M}([x],[y])+d_{M}([z],[y]).
	\end{equation}
	and similarly we have $d_{M}([y],[z])\leq d_{M}([y],[x])+d_{M}([x],[z])$.
	\item \textit{If $x\in X_\infty$, $y\in X_{j_k}$, and $z\in X_{j_l}$ for some $k,l\in\bN$ with $k\neq l$.}
	
	In this case, 
	\begin{equation}
		d_{M}([y],[z])\overset{\eqref{e.defdm}}{\leq}\rho_{k}(x,y)+\rho_{l}(x,z)\overset{\eqref{e.defdm}}{=}d_{M}([x],[y])+d_{M}([x],[z])
	\end{equation}
	and \begin{equation}
		d_{M}([x],[y])\overset{\eqref{e.defdm}}{=}\rho_{k}(x,y)\overset{\eqref{e.pun32}}{\leq}\rho_{k}(x,z)+d_{j_{k}}(y,z)\overset{\eqref{e.defdm}}{=}d_{M}([x],[z])+d_{M}([y],[z])
	\end{equation}
	and similarly we have $d_{M}([x],[z])\leq d_{M}([x],[y])+d_{M}([y],[z])$.
	\item \textit{If $x\in X_{j_k}$, $y\in X_{j_l}$ and $z\in X_{j_m}$ for some $k,l,m\in\bN$.}
	
	In this case, for any $u,v\in X_{\infty}$, we have 
	\begin{align}
		d_{M}([x],[y])&\overset{\eqref{e.defdm}}{\leq} \rho_{k}(u,x)+\rho_{l}(u,y)\\
		&\overset{\eqref{e.pun31}}{\leq} \rho_{k}(v,x)+d_{\infty}(v,u)+\rho_{l}(u,y)\\
		&\overset{\eqref{e.adm1}}{\leq}\rho_{k}(v,x)+\rho_{m}(v,z)+\rho_{m}(u,z)+\rho_{l}(u,y)
	\end{align}
	Taking the infimum over $u$ and $v$ gives $d_M([x],[y])\le d_M([x],[z])+d_M([z],[y])$. Similarly, we get $d_M([x],[z])\le d_M([x],[y])+d_M([y],[z])$ and $d_M([y],[z])\le d_M([y],[x])+d_M([x],[z])$.
\end{enumerate}
It remains to prove that $d_M$ separates distinct points. Let $[x]$, $[y]\in M$. Suppose $d_{M}([x],[y])=0$. \begin{enumerate}[label=\textit{Case {\arabic*}.},align=right,leftmargin=*,topsep=5pt,parsep=0pt,itemsep=2pt]
	\item If $x,y\in X_{\infty}$ or if $x,y\in X_{j_{k}}$ for some $k\in\bN$, then $x=y$ because $d_{\infty}$ and $d_{j_{k}}$ separate distinct points.
	\item  If $x\in X_\infty$ and $y\in X_{j_k}$, then $0=d_M(x,y)=\rho_k(x,y)$.  Choose $u_n\in B_{j_k}(p_{j_k},k)$ such that $d_\infty(x,f_k(u_n))+\lambda_k(u_n)+d_{j_k}(u_n,y)\to0$ as $n\to\infty$.  Since $\lambda_k$ takes only the values $0$ and $\eta_k$, we must have $u_n=p_{j_k}$ for all large $n$.  Hence $d_\infty(x,p_\infty)\to0$ and $d_{j_k}(p_{j_k},y)\to0$, so $x=p_\infty$, $y=p_{j_k}$ and $[x]=[p_\infty]=[p_{j_k}]=[y]$.
	\item If $x\in X_{j_k}$ and $y\in X_{j_l}$ with $k\neq l$.  By \eqref{e.defdm}, there are $z_n\in X_\infty$ such that $\rho_k(z_n,x)+\rho_l(z_n,y)\to0$ as $n\to\infty$. By the definition of $\rho_k$ in \eqref{e.rhok}, we can choose $u_n\in B_{j_k}(p_{j_k},k)$ with
\[
 \lim_{n\to\infty}(d_\infty(z_n,f_k(u_n))+\lambda_k(u_n)+d_{j_k}(u_n,x))=0.
\]
Since $\lambda_k$ takes only the values $0$ and $\eta_k$, it follows that $u_n=p_{j_k}$ for all large $n$.  Hence $d_{j_k}(p_{j_k},x)\to0$, so $x=p_{j_k}$.  The same argument applied to $\rho_l(z_n,y)\to0$ gives $y=p_{j_l}$.  Thus $[x]=[p_{j_k}]=[p_{j_l}]=[y]$ in $M$.
\qedhere
\end{enumerate}
\end{proof}

Define two maps \begin{equation}\label{e.iotak}
	\iota_{j_{k}}:X_{j_{k}}\to M,\ X_{j_{k}}\ni x\mapsto [x]\in M,
\end{equation}
and
\begin{equation}\label{e.iotai}
	\iota_{\infty}:X_{\infty}\to M,\ X_{\infty}\ni x\mapsto [x]\in M.
\end{equation}
\begin{proof}[Proof of Theorem \ref{t.AW=GH}-\ref{it.GH>AW}]
	Let $(M,d_{M},p)$ be the pointed metric space defined in Lemma \ref{l.Mdp}. Let $\iota_{j_{k}}$ and $\iota_{\infty}$ be defined by \eqref{e.iotak} and \eqref{e.iotai}, respectively. Clearly, by \eqref{e.defdm}, $\iota_{j_{k}}$ and $\iota_{\infty}$ are isometries, and \begin{equation}
	\iota_{j_{k}}(p_{j_{k}})=[p_{j_{k}}]=p\in M\ \text{ and }\ \iota_{\infty}(p_{\infty})=[p_{\infty}]=p\in M.
\end{equation}
Set $F_\infty:=\iota_\infty(X_\infty)$ and $F_k:=\iota_{j_k}(X_{j_k})$.  We prove that $F_k\to F_\infty$ locally in the Hausdorff sense.

From Lemma \ref{l.puns}-\ref{it.pun5}, 
\begin{equation}\label{eq:bridge-close-clean}
d_M\bigl(\iota_\infty(f_k(y)),\iota_{j_k}(y)\bigr)\overset{\eqref{e.defdm}}{=}\rho_k(f_k(y),y)\le\eta_k,\ \text{for every $y\in B_{j_k}(p_{j_k},k)$}.
\end{equation}

First fix $S>0$ and work in the ball $B_M(p,S)$.  If $k>S$ and $[z]\in F_k\cap B_M(p,S)$, then, because $\iota_{j_k}$ is isometric, $d_{j_k}(p_{j_k},z)<S<k$.  Hence $z\in B_{j_k}(p_{j_k},k)$.  By \eqref{eq:bridge-close-clean}, $\dist_M(z,F_\infty)\le\eta_k$.  Therefore
\begin{equation}\label{eq:Fk-to-Finf-clean}
\limsup_{k\to\infty}\sup_{z\in F_k\cap B_M(p,S)}\dist_M(z,F_\infty)\leq\lim_{k\to\infty}\eta_k=0.
\end{equation}
Conversely, let $k$ large enough such that $k-\eta_k>S$. Let $[z]\in F_\infty\cap B_M(p,S)$.  Then $z\in B_\infty(p_\infty,k-\eta_k)$.  By \eqref{e.fkdis1}, we can choose $y\in B_{j_k}(p_{j_k},k)$ such that $d_\infty(z,f_k(y))<\eta_k$. Therefore,
\[
 d_M(z,y)\overset{\eqref{e.defdm}}{=}\rho_k(z,y)\overset{\eqref{e.rhok}}{\le} d_\infty(z,f_k(y))+\lambda_k(y)<\eta_k+\eta_k=2\eta_k.
\]
Thus
\begin{equation}\label{eq:Finf-to-Fk-clean}
\limsup_{k\to\infty}\sup_{z\in F_\infty\cap B_M(p,S)}\dist_M(z,F_k)\leq \lim_{k\to\infty}2\eta_k=0.
\end{equation}
Now let $q\in M$ and $R\in(0,\infty)$ be arbitrary.  Put $S:=d_M(p,q)+R$.  Then $B_M(q,R)\subset B_M(p,S)$.  The estimates \eqref{eq:Fk-to-Finf-clean} and \eqref{eq:Finf-to-Fk-clean}, applied with this value of $S$, imply
\[
\limsup_{k\to\infty}\sup_{y\in F_k\cap B_M(q,R)}\dist_M(y,F_\infty)=0,
\ \text{ and }\
\limsup_{k\to\infty}\sup_{x\in F_\infty\cap B_M(q,R)}\dist_M(x,F_k)=0.
\]
This shows that $(X_{j_{k}},d_{j_{k}},p_{j_{k}})\xrightarrow{\mathrm{AW}}(X_{\infty},d_{\infty},p_{\infty})$.
\end{proof}
\end{appendices}
\noindent \textbf{Acknowledgments.} I am very grateful to Professor Sylvester Eriksson-Bique for valuable discussions and suggestions concerning the proof of Theorem \ref{t.main}. I thank Professor Mathav Murugan for his helpful comments and for the reference \cite{LKP04}. The author is supported by the research grant (VIL73729) from Villum Fonden.

\vspace{10pt}
\noindent Aobo Chen\orcidlink{0009-0004-0712-7941}

\vspace{3pt}
\noindent Department of Mathematics, Aarhus University, 8000 Aarhus C, Denmark
\vspace{3pt}

\noindent \texttt{aobochen.math@hotmail.com}, \texttt{aobochen@math.au.dk}

\end{document}